\documentclass{article}%
\usepackage{amsfonts}
\usepackage{amsmath}
\usepackage{amssymb}
\usepackage{graphicx}%
\newtheorem{theorem}{Theorem}

\newtheorem{definition}[theorem]{Definition}

\newtheorem{lemma}[theorem]{Lemma}
\newtheorem{notation}[theorem]{Notation}

\newtheorem{proposition}[theorem]{Proposition}
\newtheorem{remark}[theorem]{Remark}

\newenvironment{proof}[1][Proof]{\noindent\textbf{#1.} }{\ \rule{0.5em}{0.5em}}
\begin{document}

\title{On the lifting and approximation theorem \\for nonsmooth vector fields\thanks{\textbf{2000 AMS\ Classification}: Primary
53C17. \textbf{Keywords}: nonsmooth H\"{o}rmander's vector fields, Lifting,
Subelliptic distance}}
\author{Marco Bramanti, Luca Brandolini, Marco Pedroni}
\maketitle

\begin{abstract}
We prove a version of Rothschild-Stein's theorem of lifting and approximation
and some related results in the context of nonsmooth H\"{o}rmander's vector
fields for which the highest order commutators are only H\"{o}lder continuous.
The theory explicitly covers the case of one vector field having weight two
while the others have weight one.

\end{abstract}

\section*{Introduction}

This paper is focussed on the well-known \textquotedblleft lifting and
approximation\textquotedblright\ theorem proved by Rothschild-Stein in
\cite{RS}, and related topics. To describe the context and aim of this paper,
we have therefore to recall what that theorem is about. (Here we will be
rather sketchy, while precise definitions will be given later).

Let us consider a family of real smooth vector fields $X_{0},X_{1},...,X_{n}$
defined in some domain of $\mathbb{R}^{p},$ and the corresponding second order
differential operator%
\begin{equation}
L=\sum_{i=1}^{n}X_{i}^{2}+X_{0}. \label{LX_0}%
\end{equation}

If the $X_{i}$'s satisfy Hormander's condition, then $L$ is hypoellitptic
(H\"{o}rmander's theorem, \cite{H}).

If there exists in $\mathbb{R}^{p}$ a structure of \textquotedblleft
homogeneous group\textquotedblright\ such that $L$ is left invariant (with
respect to the group translations) and homogeneous of degree two (with respect
to the group dilations), then $L$ possesses a homogeneous fundamental solution
(Folland, \cite{F1}) which allows one to apply fairly standard techniques of
singular integrals, in order to prove a-priori estimates and other interesting
properties of $L.$ If such a group structure does not exists, then
Rothschild-Stein's theory tells us that it is still possible to reduce, in a
suitable sense, the study of $L$ to the study of a homogeneous left invariant
operator. This requires a three-step process. First, one \textquotedblleft
lifts\textquotedblright\ the original vector fields
\[
X_{i}=\sum_{j=1}^{p}a_{j}\left(  x\right)  \partial_{x_{j}},\text{ \ }%
x\in\mathbb{R}^{p}%
\]
which are assumed to satisfy H\"{o}rmander's condition at some step $r,$ to
some new vector fields%
\[
\widetilde{X}_{i}=X_{i}+\sum_{j=1}^{m}b_{ij}\left(  x,t\right)  \partial
_{t_{j}},\text{ \ \ }\left(  x,t\right)  \in\mathbb{R}^{p+m}%
\]
so that these lifted vector fields still satisfy H\"{o}rmander's condition at
the same step $r,$ and are free up to step $r.$

Second, one proves that in $\mathbb{R}^{p+m}$ there exists a structure of
homogeneous group $G$ and a family of left invariant homogeneous vector fields
$Y_{i}$ which locally approximate the $\widetilde{X}_{i}$'s. More precisely,
for every point $\eta=\left(  x,t\right)  $ there is a local diffeomorphism%
\[
u=\Theta_{\eta}\left(  \xi\right)
\]
from a neighborhood of $\eta$ onto a neighborhood of the origin in $G$, such
that with respect to these local coordinates,%
\[
\widetilde{X}_{i}=Y_{i}+R_{i}^{\eta}%
\]
where the \textquotedblleft remainder\textquotedblright\ $R_{i}^{\eta}$ is a
vector field, smoothly depending on the parameter $\eta,$ such that its action
on the fundamental solution $\Gamma$ of%
\[
\sum_{i=1}^{n}Y_{i}^{2}+Y_{0}%
\]
gives a function which is less singular than $Y_{i}\Gamma.$ The map
$\Theta_{\eta}\left(  \cdot\right)  ,$ a key object in this theory, also
possesses other interesting properties:

(1) it depends smoothly on $\eta;$

(2) the function%
\[
\rho\left(  \xi,\eta\right)  =\left\Vert \Theta_{\eta}\left(  \xi\right)
\right\Vert
\]
(where $\left\Vert \cdot\right\Vert $ is a homogeneous norm on the group $G$)
is a quasidistance (which also turns out to be equivalent to the distance
induced by the vector fields);

(3) the change of variables $u=\Theta_{\eta}\left(  \xi\right)  $ obeys to%
\[
d\xi=c\left(  \eta\right)  \left(  1+O\left(  \left\Vert u\right\Vert \right)
\right)  du
\]
where the function $c\left(  \eta\right)  $ is smooth and bounded away from zero.

The set of results just described allows one to prove suitable a priori
estimates for the lifted operator $\widetilde{L}.$ Once these are proved, it
is not difficult (third step) to derive the corresponding estimates for the
original operator, exploiting the fact the $L$ is the projection of
$\widetilde{L}$ on $\mathbb{R}^{p}$. All these results have been proved in
\cite{RS}.

Over the years, the lifting and approximation technique has showed to be
useful also for other purposes. In particular, sometimes the lifting theorem
is enough (without need of the approximation part of the theory), in order to
reduce problems for a general family of H\"{o}rmander's vector fields to
problems for free vector fields, which for algebraic reasons are easier to be studied.

Since the original proof of the lifting and approximation theorem given in
\cite{RS} is long and difficult, several authors have given alternative
proofs: H\"{o}rmander-Melin \cite{HM}, Folland \cite{F2} and Goodman \cite{G}
prove the lifting theorem and a pointwise version of the approximation
theorem, without dealing with the map $\Theta_{\eta}\left(  \cdot\right)  $
and its properties; Folland restricts to the particular case when the starting
vector fields are already left invariant and homogeneous with respect to a
group structure (but are not free);\ more recently, Bonfiglioli-Uguzzoni
\cite{BU} have proved that under Folland's assumptions, the original vector
fields can be lifted directly to free left invariant homogeneous vector fields
(in other words, in this case the \textquotedblleft
remainders\textquotedblright\ $R_{i}^{\eta}$ can be taken equal to zero).
Coming back to the case of general H\"{o}rmander's vector fields,
Christ-Nagel-Stein-Wainger \cite{CNSW} prove a somewhat more general version
of the lifting theorem, because they consider \textquotedblleft
weighted\textquotedblright\ vector fields (we will explain this feature in a
moment); on the other hand, they do not prove any approximation result.

Although Rothschild-Stein state their main results for a H\"{o}rmander
operator (\ref{LX_0}), all their proofs are written for the \textquotedblleft
sum of squares\textquotedblright\ operator%
\[
L=\sum_{i=1}^{n}X_{i}^{2}.
\]
The issue in handling H\"{o}rmander's operators (\ref{LX_0}) consists in the
fact that the vector field $X_{0}$ has \textquotedblleft
weight\textquotedblright\ $2,$ while $X_{1},X_{2},...,X_{n}$ have weight $1$;
this fact requires to modify in a suitable way all the basic definitions
appearing in this context (free vector fields, weight of a commutator,...);
due to the complexity of the theory, this adaptation is not trivial.
Nevertheless, as far as we know, a detailed proof of lifting, approximation,
and properties of the map $\Theta_{\eta}\left(  \cdot\right)  $, adapted to
the case of weighted vector fields has not been written yet (as we have
already pointed out, the paper \cite{CNSW} considers weighted vector fields
but only contains a proof of the lifting result).

A first aim of this paper is to present a detailed proof of the aforementioned
results, explicitly covering the case of \emph{weighted vector fields}.

Second, we are interested in extending these results to the case of
\emph{nonsmooth }vector fields, that is, vector fields which only possess the
number of derivatives involved in the commutators which are necessary to check
H\"{o}rmander's condition, with H\"{o}lder continuous derivatives of the
maximum order. This is part of a larger project which we have started in
\cite{BBP}, where we have proved in this nonsmooth context a Poincar\'{e}'s
inequality, together with the basic properties of the distance induced by the
$X_{i}$'s: Chow's connectivity theorem, the doubling condition, the
equivalence between different distances induced by the $X_{i}$'s, etc. We
refer to the introduction of \cite{BBP} for a survey of the existing
literature about nonsmooth H\"{o}rmander's vector fields.

We point out that the \textit{lifting }theorem which we prove here has already
been used in \cite{BBP}, as one of the tools in the proof of a Poincar\'{e}'s
inequality. Moreover, the whole set of results proved in this paper allows to
establish, for the operator $L,$ the existence of a local fundamental
solution. This result, obtained by a suitable adaptation of the Levi's
parametrix method, will be accomplished in the forthcoming paper \cite{BBMP}.

The main results about nonsmooth weighted vector fields proved in this paper
are: lifting (Thm. \ref{Thm Lifting}, \S 1.2), approximation (Thm.
\ref{Thm nonsmooth approx}, \S 3.3), and properties of the map $\Theta_{\eta
}\left(  \cdot\right)  $ (Prop. \ref{Prop Theta C alfa}, \S 3.3, Prop.
\ref{Prop equiv dist} and Prop. \ref{Prop nonsmooth change}, \S 3.4), which
are the analog of the properties (1), (2), (3) quoted at the beginning of the Introduction.

We now describe the general strategy that we have followed and the structure
of the paper.

In \S 1 we prove the lifting theorem for weighted nonsmooth H\"{o}rmander's
vector field. Here and in the following section we adapt and detail the
arguments in \cite{HM}. However, in order to prove the approximation result
for nonsmooth vector fields, it is not possible to proceed further in the line
of \cite{HM}. The reason is that a basic idea of this theory is that of
rewriting the vector fields in \textquotedblleft canonical
coordinates\textquotedblright; this means to apply a suitable change of
variables, which however turns out to be just H\"{o}lder continuous if the
vector fields have the limited smoothness that we assume, so that this way is closed.

Therefore, in \S 2, we pass to consider free smooth vector fields, proving for
them the approximation result, a ball-box theorem, and the desired properties
of the map $\Theta_{\eta}\left(  \cdot\right)  $. As a by-product we also get,
in the particular case of free vector fields, a quite simple proof of the
results due to Nagel-Stein-Wainger \cite{NSW} about the volume of metric balls
and the doubling condition.

Then, in \S 3, we come back to nonsmooth free vector fields. Now the natural
idea is to approximate nonsmooth vector fields with smooth ones, obtained
taking suitable Taylor's expansions of the coefficients. This idea, firstly
introduced in \cite{C}, has been used also in \cite{BBP}. To these
approximating smooth vector fields we can apply the theory developed in \S 2,
in order to derive the corresponding results in the nonsmooth case. We stress
the fact that, in the nonsmooth context, the properties of the map
$\Theta_{\eta}\left(  \cdot\right)  $ hold in a weaker form: the dependence on
$\eta$ is only H\"{o}lder continuous. However, this is enough to get the
aforementioned existence result for a fundamental solution for $L$.

\noindent\textbf{Acknowledgement.} We wish to thank Professor A. Melin, who
kindly accepted to clarify to us some arguments contained in his paper
\cite{HM}.

\section{Lifting of nonsmooth H\"{o}rmander's vector fields}

\subsection{Assumptions and notation\label{subsec assumptions}}

Let $X_{0},X_{1},...,X_{n}$ be a system of real vector fields, defined in a
domain of $\mathbb{R}^{p}.$ Let us assign to each $X_{i}$ a \textit{weight}
$p_{i}$, saying that%
\[
p_{0}=2\text{ and }p_{i}=1\text{ for }i=1,2,...n.
\]

The following standard notation, will be used throughout the paper. For any
multiindex%
\[
I=\left(  i_{1},i_{2},...,i_{k}\right)
\]
we define the \textit{weight} of $I$ as%
\[
\left\vert I\right\vert =\sum_{j=1}^{k}p_{i_{j}}.
\]
Sometimes, we will also use the (usual) \textit{length} of $I,$%
\[
\ell\left(  I\right)  =k.
\]
For any vector field $X$, we denote by ad$X$ the linear operator which maps
$Y$ to $\left[  X,Y\right]  ,$ where $Y$ is any vector field and $\left[
\cdot,\cdot\right]  $ is the Lie bracket. Now, for any multiindex $I=\left(
i_{1},i_{2},...,i_{k}\right)  $ we set:%
\[
X_{I}=X_{i_{1}}X_{i_{2}}...X_{i_{k}}%
\]
and%
\[
X_{\left[  I\right]  }=\text{ad}X_{i_{1}}\text{ad}X_{i_{2}}...\text{ad}%
X_{i_{k-1}}X_{i_{k}}=\left[  X_{i_{1}},\left[  X_{i_{2}},...\left[
X_{i_{k-1}},X_{i_{k}}\right]  ...\right]  \right]  .
\]
If $I=\left(  i_{1}\right)  ,$ then%
\[
X_{\left[  I\right]  }=X_{i_{1}}=X_{I}.
\]

As usual, $X_{\left[  I\right]  }$ can be seen either as a differential
operator or as a vector field. We will write%
\[
X_{\left[  I\right]  }f
\]
to denote the differential operator $X_{\left[  I\right]  }$ acting on a
function $f$, and
\[
\left(  X_{\left[  I\right]  }\right)  _{x}%
\]
to denote the vector field $X_{\left[  I\right]  }$ evaluated at the point $x$.

\bigskip

\textbf{Assumptions (A). }We assume that for some integer $r\geq2\ $and some
bounded domain (i.e., connected open subset) $\Omega\subset\mathbb{R}^{p}$ the
following hold:

The coefficients of the vector fields $X_{1},X_{2},...,X_{n}$ belong to
$C^{r-1}\left(  \Omega\right)  ,$ while the coefficients of $X_{0}$ belong to
$C^{r-2\,}\left(  \Omega\right)  .$ Here and in the following, $C^{k}$ stands
for the classical space of functions with continuous derivatives up to order
$k$.

These assumptions are consistent in view of the following

\begin{lemma}
Under the assumption (A) above, for any $1\leq k\leq r,$ the differential
operators%
\[
\left\{  X_{I}\right\}  _{\left\vert I\right\vert \leq k}%
\]
are well defined, and have $C^{r-k}$ coefficients. The same is true for the
vector fields $\left\{  X_{\left[  I\right]  }\right\}  _{\left\vert
I\right\vert \leq k}.$
\end{lemma}

\begin{proof}
By induction on $k$. For $k=1,$ the assertion is part of the assumption (A).
Assume the assertion holds up to $k-1$, and let
\[
I=\left(  i_{1},i_{2},...,i_{m}\right)  ,\text{ with }\left\vert I\right\vert
=k.
\]
Set $I^{\prime}=\left(  i_{2},...,i_{m}\right)  $ so that $X_{I}\left(
x\right)  =X_{i_{1}}X_{I^{\prime}}\left(  x\right)  .$ If $X_{i_{1}}$ has
weight $p_{i_{1}}$, then $\left\vert I^{\prime}\right\vert =k-p_{i_{1}}\geq0;$
by inductive assumption, $X_{I^{\prime}}$ has $C^{r-k+p_{i_{1}}}$
coefficients, hence $X_{i_{1}}X_{I^{\prime}}\left(  x\right)  $ has $C^{h}$
coefficients, with $h=\min\left(  r-k+p_{i_{1}}-1,r-p_{i_{1}}\right)  \geq
r-k,$ and we are done.
\end{proof}

\subsection{H\"{o}rmander-Melin procedure\label{subsection lifting}}

We are now going to define the concept of free vector fields. Clearly, any
vector field $X_{\left[  I\right]  }$ with $\left\vert I\right\vert \leq r$
can be rewritten explicitly as a linear combination of operators of the kind
$X_{J}$ for $\left\vert J\right\vert =\left\vert I\right\vert $:
\[
X_{\left[  I\right]  }=\sum_{J}A_{IJ}X_{J}%
\]
where $\left\{  A_{IJ}\right\}  _{\left\vert I\right\vert ,\left\vert
J\right\vert \leq r}$ is a matrix of universal constants, built exploiting
only those relations between $X_{\left[  I\right]  }$ and $X_{J}$ which hold
automatically, as a consequence of the definition of $X_{\left[  I\right]  },$
regardless of the specific properties of the vector fields $X_{0}%
,X_{1},...,X_{n}$. In particular, we see that%
\[
A_{IJ}=0\text{ if }\left\vert J\right\vert \neq\left\vert I\right\vert
\]
and
\[
A_{IJ}=\delta_{IJ}\text{ if }\left\vert J\right\vert =\left\vert I\right\vert
=1\text{.}%
\]
Also, note that if $\left\{  a_{I}\right\}  _{I\in B}$ is any finite set of
constants such that%
\begin{equation}
\sum_{I\in B}a_{I}A_{IJ}=0\text{ }\forall J\text{, \ \ then \ \ }\sum_{I\in
B}a_{I}X_{\left[  I\right]  }\equiv0 \label{anti-free}%
\end{equation}
for arbitrary vector fields $X_{0},X_{1},...,X_{n}.$ Reversing this property
we get the definition of a key concept which will be dealt with in the following:

\begin{definition}
For any positive integer $s\leq r$, we say that the vector fields $X_{0}%
,X_{1},...,X_{n}$ are free up to weight $s$ at $0$, if, for any family of
constants $\left\{  a_{I}\right\}  _{\left\vert I\right\vert \leq s}$,%
\[
\sum_{\left\vert I\right\vert \leq s}a_{I}\left(  X_{\left[  I\right]
}\right)  _{0}=0\Longrightarrow\sum_{\left\vert I\right\vert \leq s}%
a_{I}A_{IJ}=0\text{ }\forall J.
\]

\end{definition}

Comparing this definition with (\ref{anti-free}) shows that $X_{0}%
,X_{1},...,X_{n}$ are free up to weight $s$ at $0$ if, roughly speaking, the
only linear identities relating the $X_{\left[  I\right]  }$'s for $\left\vert
I\right\vert \leq s$ (at $0$) are those which hold for any possible choice of
$X_{0},X_{1},...,X_{n}$, as a consequence of the formal properties of the Lie
bracket, namely antisymmetry and Jacobi identity. However, the different
weights of the vector fields make this property not so easy to state more
explicitly. For instance, saying that $X_{0},X_{1},...,X_{n}$ are free up to
weight $1$ at $0$ just means that $X_{1},...,X_{n}$ are linearly independent
(without any requirement on $X_{0}$); saying that $X_{0},X_{1},...,X_{n}$ are
free up to weight $2$ at $0$ means that:

\begin{itemize}
\item[i)] $X_{1},...,X_{n}$ are linearly independent;

\item[ii)] $X_{0}$ is not a linear combination of vector fields of the kind
$\left[  X_{i},X_{j}\right]  $ for $i,j=1,2,...,n;$

\item[iii)] the only linear relations between the $\left[  X_{i},X_{j}\right]
$'s (for $i,j=1,2,...,n$) are those following from antisymmetry.
\end{itemize}

Clearly, for a general $s$ the explicit description of this property becomes
cumbersome\textit{.}

\begin{proposition}
\label{Proposition 2}The vector fields $X_{0},X_{1},...,X_{n}$ are free of
weight $s$ at $0$ if and only if for any family of constants $\{c_{I}%
\}_{\left\vert I\right\vert \leq s}\subset\mathbb{R}$ there exists a function
$u\in C^{\infty}\left(  \mathbb{R}^{p}\right)  $ such that $X_{I}u(0)=c_{I}$
when $\left\vert I\right\vert \leq s$.
\end{proposition}

\begin{proof}
To show that the \textquotedblleft if\textquotedblright\ condition holds, let
us suppose that $\sum_{\left\vert I\right\vert \leq s}a_{I}X_{\left[
I\right]  }(0)=0$. Then
\[
0=\sum_{\left\vert I\right\vert \leq s}a_{I}X_{\left[  I\right]  }%
u(0)=\sum_{\left\vert I\right\vert \leq s}a_{I}\sum_{\left\vert J\right\vert
\leq s}A_{IJ}X_{J}u(0)=\sum_{\left\vert I\right\vert ,\left\vert J\right\vert
\leq s}a_{I}A_{IJ}c_{J}\text{,}%
\]
and this implies that $\sum_{\left\vert I\right\vert \leq s}a_{I}A_{IJ}=0$
when $\left\vert J\right\vert \leq s$, since the $c_{J}$'s are arbitrary. Thus
$X_{0},X_{1},...,X_{n}$ are free of weight $s$ at $0$.

Now we need some notations to prove the \textquotedblleft only
if\textquotedblright\ condition. We consider polynomials in the noncommuting
variables $\xi_{0},\xi_{1},...,\xi_{n}$ and we assign to $\xi_{0}$ the weight
$p_{0}=2$, and to $\xi_{i}$, for $i=1,...,n$ the weight $p_{i}=1$. As we did
with the vector fields $X_{0},X_{1},...,X_{n}$, for any multi-index $I=\left(
i_{1},i_{2},...,i_{k}\right)  $ we put $\xi_{\left[  I\right]  }=\text{ad }%
\xi_{i_{1}}...\text{ad }\xi_{i_{k-1}}\,\xi_{i_{k}}$, where $\text{ad }\xi
_{i}\,\xi_{j}=\xi_{i}\xi_{j}-\xi_{j}\xi_{i}$. Finally, we let $V$ be the
vector space spanned by the monomials $\xi_{I}$, with $\left\vert I\right\vert
\leq s$, and $V^{\prime}$ be its dual space. Every function $u\in C^{\infty
}\left(  \mathbb{R}^{p}\right)  $ gives rise to the linear map $\Lambda_{u}\in
V^{\prime}$ defined by $\Lambda_{u}(p)=p(X_{0},...,X_{n})u(0)$, where $p\in
V$. (Notation: if $p=\xi_{I},$ then $p(X_{0},...,X_{n})u(0)=\left(
X_{I}u\right)  \left(  0\right)  $). Thus we have a mapping
\begin{align*}
\Lambda &  :C^{\infty}\left(  \mathbb{R}^{p}\right)  \rightarrow V^{\prime}\\
\Lambda &  :u\mapsto\Lambda_{u}%
\end{align*}
and our aim is to show that it is surjective. More precisely, if $L\in
V^{\prime}$ is defined by $L(\xi_{I})=c_{I}$, we have to find $u\in C^{\infty
}\left(  \mathbb{R}^{p}\right)  $ such that $\Lambda_{u}=L$. Let us denote
with $V_{j}$ the subspace of $V$ spanned by the products $\xi_{\left[
I_{1}\right]  }\cdots\xi_{\left[  I_{\nu}\right]  }$, with $\nu\leq j$ (and
$\left\vert I_{1}\right\vert +\cdots+\left\vert I_{\nu}\right\vert \leq s$).
Notice that $V_{s}=V$. We will show by induction with respect to $j$, with
$1\leq j\leq s$, that there exists $u\in C^{\infty}\left(  \mathbb{R}%
^{p}\right)  $ such that $\Lambda_{u}=L$ on $V_{j}$, that is to say,
\begin{equation}
X_{\left[  I_{1}\right]  }\cdots X_{\left[  I_{\nu}\right]  }u(0)=L(\xi
_{\left[  I_{1}\right]  }\cdots\xi_{\left[  I_{\nu}\right]  }) \label{induznu}%
\end{equation}
if $\nu\leq j$ and $\left\vert I_{1}\right\vert +\cdots+\left\vert I_{\nu
}\right\vert \leq s$.

If $j=1$, then $\nu=1$ and (\ref{induznu}) can be written as
\[
X_{\left[  I\right]  }u(0)=L\left(  \xi_{\left[  I\right]  }\right)  \text{
for any }\left\vert I\right\vert \leq s.
\]
Since the $X_{i}$'s are free of weight $s$ at $0$, we have that
\begin{equation}
\sum_{\left\vert I\right\vert \leq s}a_{I}\left(  X_{\left[  I\right]
}\right)  _{0}=0\qquad\implies\qquad\sum_{\left\vert I\right\vert \leq s}%
a_{I}\xi_{\left[  I\right]  }=0\text{.} \label{freexi}%
\end{equation}
Namely, $\sum_{\left\vert I\right\vert \leq s}a_{I}\left(  X_{\left[
I\right]  }\right)  _{0}=0$ implies that $\sum_{\left\vert I\right\vert \leq
s}a_{I}A_{IJ}=0$ for any $J,$ hence%
\[
\sum_{\left\vert I\right\vert \leq s}a_{I}\xi_{\left[  I\right]  }%
=\sum_{\left\vert I\right\vert \leq s}a_{I}\sum_{J}A_{IJ}\xi_{J}=\sum
_{J}\left(  \sum_{\left\vert I\right\vert \leq s}a_{I}A_{IJ}\right)  \xi
_{J}=0.
\]
By (\ref{freexi}), there is a (unique) linear form defined on the span of the
tangent vectors $\{\left(  X_{\left[  I\right]  }\right)  _{0}\}_{\left\vert
I\right\vert \leq s}$ by
\[
\left(  X_{\left[  I\right]  }\right)  _{0}\mapsto L\left(  \xi_{\left[
I\right]  }\right)  \text{.}%
\]
We can extend this form to $\mathbb{R}^{p}$ and then find a function $u\in
C^{\infty}\left(  \mathbb{R}^{p}\right)  $, e.g., a first degree homogeneous
polynomial, such that the differential $u^{(1)}(0)=du(0)$ of $u$ at 0
coincides with such an extension. Since $u^{(1)}(0)\left(  X_{\left[
I\right]  }\right)  _{0}=X_{\left[  I\right]  }u(0)$, the case $j=1$ is done.

Assume now that for any $L\in V^{\prime}$ there exists $u_{0}\in C^{\infty
}\left(  \mathbb{R}^{p}\right)  $ such that $\Lambda_{u_{0}}=L$ on $V_{j-1}$.
This means that
\begin{equation}
X_{\left[  I_{1}\right]  }\cdots X_{\left[  I_{\nu}\right]  }u_{0}%
(0)=L(\xi_{\left[  I_{1}\right]  }\cdots\xi_{\left[  I_{\nu}\right]  })
\label{induzj-1}%
\end{equation}
when $\nu\leq j-1$ and $\left\vert I_{1}\right\vert +\cdots+\left\vert I_{\nu
}\right\vert \leq s$. If $u=u_{0}+v$, with $v$ vanishing of order $j$ at 0 (in
the usual sense), then $\Lambda_{u}=L$ on $V_{j-1}$, meaning that we must find
$v$ in such a way that (\ref{induznu}) is solved for $\nu=j$. In this case,
the equation takes the form
\begin{equation}
v^{(j)}(0)\left(  \left(  X_{\left[  I_{1}\right]  }\right)  _{0}%
,\dots,\left(  X_{\left[  I_{j}\right]  }\right)  _{0}\right)  =L(\xi_{\left[
I_{1}\right]  }\cdots\xi_{\left[  I_{j}\right]  })-X_{\left[  I_{1}\right]
}\cdots X_{\left[  I_{j}\right]  }u_{0}(0)\text{,} \label{induzj}%
\end{equation}
where $v^{(j)}(0)$ is the $j$-th differential of $v$ at 0, seen as a
$j$-linear form on $\mathbb{R}^{p}$. Namely,%
\[
X_{\left[  I_{1}\right]  }\cdots X_{\left[  I_{j}\right]  }u(0)=X_{\left[
I_{1}\right]  }\cdots X_{\left[  I_{j}\right]  }u_{0}(0)+X_{\left[
I_{1}\right]  }\cdots X_{\left[  I_{j}\right]  }v(0)=L(\xi_{\left[
I_{1}\right]  }\cdots\xi_{\left[  I_{\nu}\right]  })
\]
but $X_{\left[  I_{1}\right]  }\cdots X_{\left[  I_{j}\right]  }v(0)$ simply
equals $v^{(j)}(0)\left(  \left(  X_{\left[  I_{1}\right]  }\right)
_{0},\dots,\left(  X_{\left[  I_{j}\right]  }\right)  _{0}\right)  ,$ because
all the derivatives of $v$ of intermediate order (which appears expanding the
differential operator $X_{\left[  I_{1}\right]  }\cdots X_{\left[
I_{j}\right]  }$) actually vanish because $v$ vanishes of order $j$ at 0.

Thus, if we show that the right-hand side of (\ref{induzj}) defines a
symmetric $j$-linear form on the span of the tangent vectors $\left(
X_{\left[  I\right]  }\right)  _{0}$, where $\left\vert I\right\vert \leq s,$
then we are done, because we can then extend this form to $\mathbb{R}^{p}$ and
therefore find a function $v\in C^{\infty}\left(  \mathbb{R}^{p}\right)  $
vanishing of order $j$ at 0, e.g., a $j$-th degree homogeneous polynomial,
such that its $j$-th differential $v^{(j)}(0)$ at 0 coincides with the
extended $j$-linear form.

So, let $J$ be the form defined by%
\[
J:\left(  \left(  X_{\left[  I_{1}\right]  }\right)  _{0},\dots,\left(
X_{\left[  I_{j}\right]  }\right)  _{0}\right)  \mapsto L(\xi_{\left[
I_{1}\right]  }\cdots\xi_{\left[  I_{j}\right]  })-X_{\left[  I_{1}\right]
}\cdots X_{\left[  I_{j}\right]  }u_{0}(0).
\]
The check that this $j$-linear form is actually well defined amounts to show
that%
\begin{align*}
&  \sum a_{I_{1}}\left(  X_{\left[  I_{1}\right]  }\right)  _{0}=0,\sum
a_{I_{2}}\left(  X_{\left[  I_{2}\right]  }\right)  _{0}=0,...,\sum a_{I_{j}%
}\left(  X_{\left[  I_{j}\right]  }\right)  _{0}=0\Longrightarrow\\
&  \sum a_{I_{1}}a_{I_{2}}...a_{I_{j}}\left\{  L(\xi_{\left[  I_{1}\right]
}\cdots\xi_{\left[  I_{j}\right]  })-X_{\left[  I_{1}\right]  }\cdots
X_{\left[  I_{j}\right]  }u_{0}(0)\right\}  =0
\end{align*}
This is almost the same as in the $j=1$ case. Indeed, the implication
(\ref{freexi}) still holds, and therefore $\sum a_{I_{i}}\left(  X_{\left[
I_{i}\right]  }\right)  _{0}=0\Longrightarrow\sum a_{I_{i}}\xi_{\left[
I_{i}\right]  }=0$ for $i=1,2,...,n$; hence%
\begin{align*}
&  \sum a_{I_{1}}a_{I_{2}}...a_{I_{j}}\left\{  L(\xi_{\left[  I_{1}\right]
}\cdots\xi_{\left[  I_{j}\right]  })-X_{\left[  I_{1}\right]  }\cdots
X_{\left[  I_{j}\right]  }u_{0}(0)\right\}  =\\
&  =L\left(  \sum a_{I_{1}}\xi_{\left[  I_{1}\right]  }\sum a_{I_{2}}%
\xi_{\left[  I_{2}\right]  }...\sum a_{Ij}\xi_{\left[  Ij\right]  }\right)
+\\
&  -\sum a_{I_{1}}X_{\left[  I_{1}\right]  }\sum a_{I_{2}}X_{\left[
I_{2}\right]  }...\sum a_{Ij}X_{\left[  I_{j}\right]  }u_{0}\left(  0\right)
=0.
\end{align*}
To show the symmetry of $J$, let us introduce
\[
d_{I_{1},\dots,I_{j}}=L(\xi_{\left[  I_{1}\right]  }\cdots\xi_{\left[
I_{j}\right]  })-X_{\left[  I_{1}\right]  }\cdots X_{\left[  I_{j}\right]
}u_{0}(0)\text{,}%
\]
and let us prove that they are symmetric in the (multi-)indices. We first need
to show that for every pair of multi-indices $I$ and $J$, one has
\[
\lbrack\xi_{\left[  I\right]  },\xi_{\left[  J\right]  }]=\sum_{\left\vert
K\right\vert =\left\vert I\right\vert +\left\vert J\right\vert }b_{K}%
\xi_{\left[  K\right]  }\text{,}%
\]
where the $b_{K}$'s are absolute constants, only depending on the multiindices
$I,J,K$. This is just a consequence of Jacoby identity, as we can show by
induction on $\ell\left(  I\right)  $. First, if $\ell\left(  I\right)  =1,$
that is $I=\left(  i\right)  ,$ there is nothing to prove, because%
\[
\lbrack\xi_{\left[  I\right]  },\xi_{\left[  J\right]  }]=[\xi_{i}%
,\xi_{\left[  J\right]  }]=\xi_{\left[  K\right]  }%
\]
with $K=\left(  i,J\right)  $, just by definition of $\xi_{\left[  K\right]
}$. Assume then the property for $\ell\left(  I\right)  \leq k,$ and let
$I=\left(  i,I^{\prime}\right)  $ with \ $\ell\left(  I^{\prime}\right)  =k$;
then%
\begin{align*}
\left[  \xi_{\left[  I\right]  },\xi_{\left[  J\right]  }\right]   &  =\left[
\xi_{\left[  i,I^{\prime}\right]  },\xi_{\left[  J\right]  }\right]  =\left[
\left[  \xi_{i},\xi_{\left[  I^{\prime}\right]  }\right]  ,\xi_{\left[
J\right]  }\right]  =\\
&  =\left[  \xi_{i},\left[  \xi_{\left[  I^{\prime}\right]  },\xi_{\left[
J\right]  }\right]  \right]  +\left[  \xi_{\left[  I^{\prime}\right]
},\left[  \xi_{\left[  J\right]  },\xi_{i}\right]  \right]  .
\end{align*}
Now the first term in the last sum is already in the proper form, while the
second can be rewritten in the proper form by inductive assumption, so we are done.

Let us show now the desired symmetry result. It is clearly sufficient to show
it for consecutive indices. We limit ourselves to verify the symmetry with
respect to the first two indices, the other cases being a straightforward
generalization of it. Indeed, we have that
\begin{align*}
d_{I_{1},I_{2},\dots,I_{j}}  &  -d_{I_{2},I_{1},\dots,I_{j}}=L([\xi_{\left[
I_{1}\right]  },\xi_{\left[  I_{2}\right]  }]\cdots\xi_{\left[  I_{j}\right]
})-[X_{\left[  I_{1}\right]  },X_{\left[  I_{2}\right]  }]\cdots X_{\left[
I_{j}\right]  }u_{0}(0)\\
&  =\sum_{\left\vert K\right\vert =\left\vert I\right\vert +\left\vert
J\right\vert }b_{K}\left(  L(\xi_{\left[  K\right]  }\xi_{\left[
I_{3}\right]  }\cdots\xi_{\left[  I_{j}\right]  })-X_{\left[  K\right]
}X_{\left[  I_{3}\right]  }\cdots X_{\left[  I_{j}\right]  }u_{0}(0)\right)
=0
\end{align*}
by the induction hypothesis (\ref{induzj-1}). Now we are (almost) done,
because we have shown that the right-hand side of (\ref{induzj}) defines a
symmetric $j$-linear form on the span of $\{\left(  X_{\left[  I\right]
}\right)  _{0})\}_{\left\vert I\right\vert \leq s}$. As we did for $j=1$, we
can extend this form to $\mathbb{R}^{p}$ and then find a function $v\in
C^{\infty}\left(  \mathbb{R}^{p}\right)  $ vanishing of order $j$ at 0, e.g.,
a $j$-th degree homogeneous polynomial, such that its $j$-th differential
$v^{(j)}(0)$ at 0 coincides with the extended $j$-linear form. This completes
the proof.
\end{proof}

\begin{proposition}
\label{Proposition 3}Let $X_{0},X_{1},...,X_{n}$ be free of weight $s-1$ but
not of weight $s$ at $0$. Then one can find vector fields $\widetilde{X}_{j}$
in $\mathbb{R}^{p+1}$ of the form%
\begin{equation}
\widetilde{X}_{j}=X_{j}+u_{j}\left(  x\right)  \frac{\partial}{\partial
t}\text{ }\left(  j=0,1,...,n\right)  \label{Xtilde}%
\end{equation}
with $u_{j}\in C^{\infty}\left(  \mathbb{R}^{p}\right)  ,$ such that

1. the $\widetilde{X}_{j}$'s remain free of weight $s-1$;

2. for every $r\geq s,$%
\[
\dim\left\langle \left(  \widetilde{X}_{\left[  I\right]  }\right)
_{0}\right\rangle _{\left\vert I\right\vert \leq r}=\dim\left\langle \left(
X_{\left[  I\right]  }\right)  _{0}\right\rangle _{\left\vert I\right\vert
\leq r}+1
\]
where the symbol $\left\langle Y_{\alpha}\right\rangle _{\alpha\in B}$ denotes
the vector space spanned by the vectors $\left\{  Y_{\alpha}:\alpha\in
B\right\}  .$
\end{proposition}

\begin{proof}
Let us show that condition 1 in the above statement holds for any choice of
the functions $u_{j}\left(  x\right)  .$ To see this, we first claim that
(\ref{Xtilde}) implies%
\begin{equation}
\widetilde{X}_{\left[  I\right]  }=X_{\left[  I\right]  }+u_{I}\left(
x\right)  \frac{\partial}{\partial t} \label{X_I}%
\end{equation}
for any multiindex $I$ and some $u_{I}\in C^{\infty}\left(  \mathbb{R}%
^{p}\right)  .$ Namely, we can proceed by induction on $\ell\left(  I\right)
.$ For $\ell\left(  I\right)  =1$, this is just (\ref{Xtilde}); assume
(\ref{X_I}) holds for $\ell\left(  I\right)  =j-1.$ For $\ell\left(  I\right)
=j,$ let $I=\left(  i,J\right)  $ for some $i=0,1,...,n$ and $\ell\left(
J\right)  =j-1.$ Then, by inductive assumption,
\begin{align*}
\widetilde{X}_{\left[  I\right]  }  &  =\widetilde{X}_{\left[  i,J\right]
}=\left[  \widetilde{X}_{i},\widetilde{X}_{\left[  J\right]  }\right]  =\\
&  =\left[  X_{i}+u_{i}\left(  x\right)  \frac{\partial}{\partial
t},X_{\left[  J\right]  }+u_{J}\left(  x\right)  \frac{\partial}{\partial
t}\right]  =\\
&  =\left[  X_{i},X_{\left[  J\right]  }\right]  +\left(  X_{i}u_{J}%
-X_{\left[  J\right]  }u_{i}\right)  \frac{\partial}{\partial t}=\\
&  =X_{\left[  I\right]  }+u_{I}\left(  x\right)  \frac{\partial}{\partial t}.
\end{align*}
Next, we show that (\ref{X_I}) implies that the $\widetilde{X}_{i}$'s are free
of weight $s-1$. If%
\[
\sum_{\left\vert I\right\vert \leq s-1}a_{I}\left(  \widetilde{X}_{\left[
I\right]  }\right)  _{0}=0
\]
for some coefficients $a_{I},$ then by (\ref{X_I}) we have%
\begin{align*}
0  &  =\sum_{\left\vert I\right\vert \leq s-1}a_{I}\left(  X_{\left[
I\right]  }+u_{I}\left(  x\right)  \frac{\partial}{\partial t}\right)  _{0}=\\
&  =\sum_{\left\vert I\right\vert \leq s-1}a_{I}\left(  X_{\left[  I\right]
}\right)  _{0}+\left(  \sum_{\left\vert I\right\vert \leq s-1}a_{I}%
u_{I}\left(  0\right)  \right)  \frac{\partial}{\partial t}.
\end{align*}
Since $\frac{\partial}{\partial t}$ is independent from the vectors $\left(
X_{\left[  I\right]  }\right)  _{0},$ this implies that $\sum_{\left\vert
I\right\vert \leq s-1}a_{I}u_{I}\left(  0\right)  =0$ and%
\[
\sum_{\left\vert I\right\vert \leq s-1}a_{I}\left(  X_{\left[  I\right]
}\right)  _{0}=0.
\]
But the $X_{i}$'s are free of weight $s-1$ at $0$, hence
\[
\sum_{\left\vert I\right\vert \leq s-1}a_{I}A_{IJ}=0\text{ for any }J\text{
with }\left\vert J\right\vert \leq s-1.
\]
Therefore also the $\widetilde{X}_{i}$'s are free of weight $s-1$ at $0.$

We now show that it is possible to choose smooth functions $u_{j}$ such that
condition 2 in the statement of this proposition holds. To show this, we will
prove that there exist functions $u_{j}$ and constants $\left\{
a_{I}\right\}  _{\left\vert I\right\vert \leq s}$ such that:%
\begin{equation}
\sum_{\left\vert I\right\vert \leq s}a_{I}\left(  X_{\left[  I\right]
}\right)  _{0}=0 \label{1}%
\end{equation}%
\begin{equation}
\sum_{\left\vert I\right\vert \leq s}a_{I}\left(  \widetilde{X}_{\left[
I\right]  }\right)  _{0}\neq0 \label{2}%
\end{equation}
From (\ref{1})-(\ref{2}), condition 2 will follow; namely,%
\[
0\neq\sum_{\left\vert I\right\vert \leq s}a_{I}\left(  \widetilde{X}_{\left[
I\right]  }\right)  _{0}=\sum_{\left\vert I\right\vert \leq s}a_{I}\left(
\left(  X_{\left[  I\right]  }\right)  _{0}+u_{I}\left(  0\right)
\frac{\partial}{\partial t}\right)  =\left(  \sum_{\left\vert I\right\vert
\leq s}a_{I}u_{I}\left(  0\right)  \right)  \frac{\partial}{\partial t}%
=b\frac{\partial}{\partial t}%
\]
with $b\neq0,$ hence%
\[
\frac{\partial}{\partial t}=\sum_{\left\vert I\right\vert \leq s}\frac{a_{I}%
}{b}\left(  \widetilde{X}_{\left[  I\right]  }\right)  _{0}%
\]
and this shows that%
\[
\left\langle \left(  \widetilde{X}_{\left[  I\right]  }\right)  _{0}%
\right\rangle _{\left\vert I\right\vert \leq r}=\left\langle \left(
X_{\left[  I\right]  }\right)  _{0}\right\rangle \oplus\left\langle
\frac{\partial}{\partial t}\right\rangle ,
\]
which implies condition 2.

To prove (\ref{1})-(\ref{2}), we use our assumption on the $X_{i}$: since they
are \textit{not} free of weight $s,$ there exist coefficients $\left\{
a_{I}\right\}  _{\left\vert I\right\vert \leq s}$ such that (\ref{1}) holds
but
\begin{equation}
\sum_{\left\vert I\right\vert \leq s}a_{I}A_{IJ}\neq0\text{ for some }J\text{
with }\left\vert J\right\vert \leq s. \label{A_IJ}%
\end{equation}
It remains to prove that there exist functions $u_{j}$ such that (\ref{2})
holds for these $u_{j}$'s and $a_{I}$'s. To determine these $u_{j}$'s, let us
examine the action of the vector field%
\[
\sum_{\left\vert I\right\vert \leq s}a_{I}\widetilde{X}_{\left[  I\right]
}=\sum_{\left\vert I\right\vert \leq s}a_{I}\sum_{\left\vert J\right\vert \leq
s}A_{IJ}\widetilde{X}_{J}%
\]
on the function $t$. For any $J$ with $\left\vert J\right\vert \leq s,$ let us
write $J=\left(  J^{\prime}j\right)  $ for some $j=0,1,...,n.$ Then%
\[
\widetilde{X}_{J}t=\widetilde{X}_{J^{\prime}}\widetilde{X}_{j}t=\widetilde
{X}_{J^{\prime}}\left[  \left(  X_{j}+u_{j}\frac{\partial}{\partial t}\right)
t\right]  =\widetilde{X}_{J^{\prime}}u_{j}=X_{J^{\prime}}u_{j}%
\]
since $u_{j}$ does not depend on $t$. We then have:%
\[
\left(  \sum_{\left\vert I\right\vert \leq s}a_{I}\widetilde{X}_{\left[
I\right]  }\left(  t\right)  \right)  \left(  0\right)  =\sum_{\left\vert
I\right\vert \leq s}a_{I}\sum_{\left\vert J\right\vert \leq s}A_{IJ}%
\sum_{j=0,...,n;\text{ }J=\left(  J^{\prime}j\right)  }\left(  X_{J^{\prime}%
}u_{j}\right)  \left(  0\right)  .
\]
Since $J=\left(  J^{\prime}j\right)  ,$
\[
\left\vert J^{\prime}\right\vert =\left\{
\begin{array}
[c]{l}%
\left\vert J\right\vert -1\text{ for }j=1,2,...,n\\
\left\vert J\right\vert -2\text{ for }j=0
\end{array}
\right.
\]
hence, in any case, $\left\vert J\right\vert \leq s$ implies $\left\vert
J^{\prime}\right\vert \leq s-1.$ Since the $X_{i}$'s are free of weight $s-1$
at $0$, by Proposition 2 for any choice of constants $\left\{  c_{J^{\prime}%
}\right\}  _{\left\vert J^{\prime}\right\vert \leq s-1}$ there exists a
function $u\in C^{\infty}\left(  \mathbb{R}^{p}\right)  $ such that $\left(
X_{J^{\prime}}u\right)  \left(  0\right)  =c_{J^{\prime}}.$ On the other hand,
by (\ref{A_IJ}), there exists a set of constants $\left\{  c_{J}\right\}
_{\left\vert J\right\vert \leq s}$ such that
\[
\sum_{\left\vert I\right\vert \leq s}\sum_{\left\vert J\right\vert \leq
s}a_{I}A_{IJ}c_{J}\neq0.
\]
Setting $c_{J^{\prime}}^{j}=c_{J}$ if $J=\left(  J^{\prime}j\right)  $ and
applying $n+1$ times Proposition 2 to the $n+1$ sets of constants $\left\{
c_{J^{\prime}}^{j}\right\}  _{\left\vert J^{\prime}\right\vert \leq s-1},$
$j=0,1,2,...,n,$ we find $u_{0},u_{1},...,u_{n}$ such that
\[
\left(  \sum_{\left\vert I\right\vert \leq s}a_{I}\widetilde{X}_{\left[
I\right]  }\left(  t\right)  \right)  \left(  0\right)  =\sum_{\left\vert
I\right\vert \leq s}\sum_{\left\vert J\right\vert \leq s}a_{I}A_{IJ}c_{J}%
\neq0.
\]
Hence (\ref{2}) holds. This completes the proof of the proposition.
\end{proof}

\begin{theorem}
[Lifting]\label{Thm Lifting}Let $X_{0},X_{1},...,X_{n}$ be vector fields
satisfying H\"{o}rmander's condition of step $r$ at $x=0$: the vectors
\[
\left\{  \left(  X_{\left[  I\right]  }\right)  _{0}\right\}  _{\left\vert
I\right\vert \leq r}%
\]
span $\mathbb{R}^{p}$. (This clearly implies that such property holds in a
suitable neighborhood of $0$). Then there exist an integer $m$ and vector
fields $\widetilde{X}_{k}$ in $\mathbb{R}^{p+m}$, of the form%
\[
\widetilde{X}_{k}=X_{k}+\sum_{j=1}^{m}u_{kj}\left(  x,t_{1},t_{2}%
,...,t_{j-1}\right)  \frac{\partial}{\partial t_{j}}%
\]
$\left(  k=0,1,...,n\right)  ,$ where the $u_{kj}$'s are polynomials, such
that the $\widetilde{X}_{k}$'s are free of weight $r$ and $\left\{  \left(
\widetilde{X}_{\left[  I\right]  }\right)  _{0}\right\}  _{\left\vert
I\right\vert \leq r}$ span $\mathbb{R}^{p+m}.$
\end{theorem}

This theorem has an obvious reformulation in any point $x_{0}\in\mathbb{R}%
^{p},$ with the lifted vector fields defined in a neighborhood of $\left(
x_{0},0\right)  \in\mathbb{R}^{p+m}.$

\begin{proof}
Let $\left\{  \left(  X_{\left[  I\right]  }\right)  _{0}\right\}  _{I\in B}$
be a basis of $\mathbb{R}^{p},$ for some set $B$ of $p$ multiindices of weight
$\leq r.$ We claim that%
\[
\text{rank}\left[  A_{IJ}\right]  _{I\in B,\left\vert J\right\vert \leq r}\geq
p,
\]
because the $p$ vectors $\left(  A_{IJ}\right)  _{\left\vert J\right\vert \leq
r}$, with $I\in B$, are independent. Namely, if for some constants $\left\{
a_{I}\right\}  _{I\in B}$
\[
\sum_{I\in B}a_{I}A_{IJ}=0\text{ for any }J\text{ with }\left\vert
J\right\vert \leq r,
\]
then%
\[
\sum_{I\in B}a_{I}\left(  X_{\left[  I\right]  }\right)  _{0}=\sum_{I\in
B}a_{I}\sum_{\left\vert J\right\vert \leq r}A_{IJ}\left(  X_{J}\right)
_{0}=\sum_{\left\vert J\right\vert \leq r}\left(  \sum_{I\in B}a_{I}%
A_{IJ}\right)  \left(  X_{J}\right)  _{0}=0.
\]
But $\left\{  \left(  X_{\left[  I\right]  }\right)  _{0}\right\}  _{I\in B}$
is a basis, hence $\sum_{I\in B}a_{I}\left(  X_{\left[  I\right]  }\right)
_{0}=0$ implies $a_{I}=0$ for any $I\in B$.

The relation just proved means that%
\begin{equation}
p\leq\text{rank}\left[  A_{IJ}\right]  _{\left\vert I\right\vert ,\left\vert
J\right\vert \leq r}\equiv c\left(  r,n\right)  , \label{p<c}%
\end{equation}
an absolute constant only depending on $r,n$.

Now, let $s\leq r$ be such that $X_{0},X_{1},...,X_{n}$ are free of weight
$s-1$ but not of weight $s$, at $0$. (If the $X_{i}$'s were already free of
weight $r$, there would be nothing to prove. We also agree to say that the
$X_{i}$'s are free of weight $0$ if they are not free of weight $1$). We can
then apply Proposition \ref{Proposition 3} and build vector fields%
\[
\widetilde{X}_{j}=X_{j}+u_{j}\left(  x\right)  \frac{\partial}{\partial
t}\text{ }\left(  j=0,1,...,n\right)
\]
in $\mathbb{R}^{p+1},$ free of weight $s-1$ and such that%
\[
\dim\left\langle \left(  \widetilde{X}_{\left[  I\right]  }\right)
_{0}\right\rangle _{\left\vert I\right\vert \leq r}=\dim\left\langle \left(
X_{\left[  I\right]  }\right)  _{0}\right\rangle _{\left\vert I\right\vert
\leq r}+1=p+1
\]
(because by assumption the $\left\{  \left(  X_{\left[  I\right]  }\right)
_{0}\right\}  _{\left\vert I\right\vert \leq r}$ span $\mathbb{R}^{p}$). Hence
the $\left\{  \left(  \widetilde{X}_{\left[  I\right]  }\right)  _{0}\right\}
_{\left\vert I\right\vert \leq r}$ still span the whole space $\mathbb{R}%
^{p+1}.$ Now: either the vector fields $\left\{  \left(  \widetilde
{X}_{\left[  I\right]  }\right)  _{0}\right\}  _{\left\vert I\right\vert \leq
r}$ are free of weight $r$, and we are done, or the assumptions of Proposition
\ref{Proposition 3} are still satisfied, and we can iterate our argument; in
this case, by (\ref{p<c}), condition $p+1\leq c\left(  r,n\right)  $ must
hold. After a suitable finite number $m$ of iterations, condition $p+m\leq
c\left(  r,n\right)  $ cannot hold anymore, and this means that the vector
fields $\widetilde{X}_{j}$ must be free of weight $r$. The iterative argument
also shows that the $u_{kj}$'s are polynomials only depending on the variables
$x,t_{1},t_{2},...,t_{j-1}$.
\end{proof}

\section{Approximation of free smooth vector
fields\label{section approximation}}

In this section we carry out the second part of Rothschild-Stein's procedure,
that is, the approximation of free vector fields by left invariant vector
fields on a homogeneous group. Here we concentrate on \emph{smooth }vector
fields, while the nonsmooth theory will be treated in Section
\ref{Section nonsmooth}. We actually prove a somewhat more general result than
that by Rothschild-Stein, in the line of \cite{HM}. In \S 2.5, we will also
prove, for free vector fields, a ball-box theorem and the resulting estimate
on the volume of metric balls, in the spirit of Nagel-Stein-Wainger's results.

By the lifting theorem (Theorem \ref{Thm Lifting}), starting from any system
of smooth H\"{o}rmander's vector fields in $\mathbb{R}^{p}$ we can define new
vector fields $\widetilde{X}_{0},\ldots,\widetilde{X}_{n}$ in a neighborhood
of $0\in\mathbb{R}^{p+m}$, free up to weight $r$ at $0$ and such that
$\left\{  \widetilde{X}_{\left[  I\right]  }\left(  0\right)  \right\}
_{\left\vert I\right\vert \leq s}$ spans $\mathbb{R}^{p+m}.$ Just to simplify
notation, throughout this section we will keep calling $X_{i}$ and
$\mathbb{R}^{p}$ these lifted free vector fields and their underlying space, respectively.

\bigskip

Therefore, let now $X_{0},X_{1},...,X_{n}$ be a system of \emph{smooth}
H\"{o}rmander's vector fields in $\Omega\subset\mathbb{R}^{p}$, \emph{free up
to weight} $r$ and satisfying H\"{o}rmander's condition of step $r$ in
$\Omega$. Since the $X_{i}$'s are free, it is possible to choose a set $B$ of
$p$ multiindices $I$ with $\left\vert I\right\vert \leq r,\ $such that
$\left\{  X_{\left[  I\right]  }\right\}  _{I\in B}$ is a basis of
$\mathbb{R}^{p}$ at any point $x\in\Omega.$ We assume this set $B$ fixed once
and for all.

\subsection{Canonical coordinates and weights of vector
fields\label{subsection canonical coordinates}}

Let us recall the standard definition of \textit{exponential of a vector
field}. We set:%
\[
\exp\left(  tX\right)  \left(  \overline{x}\right)  =\varphi\left(  t\right)
\]
where $\varphi$ is the solution to the Cauchy problem%
\begin{equation}
\left\{
\begin{array}
[c]{l}%
\varphi^{\prime}\left(  \tau\right)  =X_{\varphi\left(  \tau\right)  }\\
\varphi\left(  0\right)  =\overline{x}%
\end{array}
\right.  \label{Cauchy2}%
\end{equation}
The point $\exp\left(  tX\right)  \left(  \overline{x}\right)  $ is uniquely
defined for $t\in\mathbb{R}$ small enough, as soon as $X$ has Lipschitz
continuous coefficients, by the classical Cauchy's theorem about existence and
uniqueness for solutions to Cauchy problems. For a fixed $\Omega^{\prime
}\Subset\Omega,$ a $t$-neighborhood of zero where $\exp\left(  tX\right)
\left(  \overline{x}\right)  $ is defined can be found uniformly for
$\overline{x}$ ranging in $\Omega^{\prime}$.

Equivalently, we can write%
\[
\exp\left(  tX\right)  \left(  \overline{x}\right)  =\phi\left(  1\right)
\]
where $\phi$ is the solution to the Cauchy problem%
\[
\left\{
\begin{array}
[c]{l}%
\phi^{\prime}\left(  \tau\right)  =tX_{\phi\left(  \tau\right)  }\\
\phi\left(  0\right)  =\overline{x}.
\end{array}
\right.
\]

Now, for any $\overline{x}\in\Omega,$ let us introduce the set of local
(\textquotedblleft canonical\textquotedblright) coordinates%
\begin{equation}
\mathbb{R}^{p}\backepsilon u\longmapsto\exp\left(  \sum_{I\in B}%
u_{I}X_{\left[  I\right]  }\right)  \left(  \overline{x}\right)  ,
\label{can-coos}%
\end{equation}
defined for $u$ in a suitable neighborhood of $0.$ Note that the Jacobian of
the map $u\longmapsto x,$ at $u=0,$ equals the matrix of the vector fields
$\left\{  \left(  X_{\left[  I\right]  }\right)  _{\overline{x}}\right\}
_{I\in B},$ therefore is nonsingular. This allows to define canonical
coordinates in a suitable neighborhood $U\left(  \overline{x}\right)  $ of
$\overline{x}$.

Since the basis $\left\{  \left(  X_{\left[  I\right]  }\right)
_{\overline{x}}\right\}  _{I\in B}$ depends continuously on the point
$\overline{x},$ the radius of this neighborhood can be taken uniformly bounded
away from zero for $\overline{x}$ ranging in a compact set.

Henceforth in this section, all the computation will be made with respect to
this system of coordinates defined in a neighborhood of the point
$\overline{x}$ (which has canonical coordinates $u=0$).

Our aim is to establish some basic properties enjoyed by the vector fields
$X_{\left[  I\right]  },$ if they are expressed with respect to canonical
coordinates, in particular Theorem \ref{Thm weight of X_[I]}, which will be a
key tool for the following.

We start with the following:

\begin{lemma}
If we express the vector fields $X_{\left[  I\right]  }$ with respect to the
above coordinates $u$, then we have that
\begin{equation}
\sum_{I\in B}u_{I}\frac{\partial}{\partial u_{I}}=\sum_{I\in B}u_{I}X_{\left[
I\right]  }. \label{Euler}%
\end{equation}
(In the following, we will also write $e_{I}$ for $\frac{\partial}{\partial
u_{I}}$).
\end{lemma}

\begin{proof}
We start noting that, If $Y=\sum_{I\in B}y_{I}\left(  u\right)  \frac
{\partial}{\partial u_{I}}$ and $Z=\sum_{I\in B}z_{I}\left(  u\right)
\frac{\partial}{\partial u_{I}}$ are two vector fields such that%
\[
Z\left(  u_{J}\right)  =Y\left(  u_{J}\right)  \text{ for any }J\in B,
\]
(that is, the vector fields act at the same way on the functions $u\mapsto
u_{J}$) then $y_{I}\left(  u\right)  =z_{I}\left(  u\right)  $ for any $I\in
B,$ hence $Y=Z$. Therefore, it will be enough to show that%
\[
\left(  \sum_{I\in B}u_{I}X_{\left[  I\right]  }\right)  \left(  u_{J}\right)
=\left(  \sum_{I\in B}u_{I}\frac{\partial}{\partial u_{I}}\right)  \left(
u_{J}\right)  \text{.}%
\]
Now, for any vector field $Y,$%
\[
Yf\left(  x\right)  =\frac{d}{dt}\left(  f\left(  \exp\left(  tY\right)
\left(  \overline{x}\right)  \right)  \right)  _{/t=t_{0}}\text{ where }%
x=\exp\left(  t_{0}Y\right)  \left(  c\right)  \text{.}%
\]
Hence, if $Y=\sum_{I\in B}u_{I}X_{\left[  I\right]  }$, then
\[
\left(  \sum_{I\in B}u_{I}X_{\left[  I\right]  }\right)  \left(  u_{J}\right)
=\frac{d}{dt}\left(  u_{J}\left(  \exp\left(  t\sum_{I\in B}u_{I}X_{\left[
I\right]  }\right)  \left(  \overline{x}\right)  \right)  \right)  _{/t=t_{0}}%
\]
just by definition of the coordinates $u_{I}$%
\[
=\frac{d}{dt}\left(  tu_{J}\right)  _{/t=t_{0}}=u_{J}=\left(  \sum_{I\in
B}u_{I}\frac{\partial}{\partial u_{I}}\right)  \left(  u_{J}\right)  .
\]

\end{proof}

\begin{definition}
[Weights]\label{Definition weights}We now assign the weight $\left\vert
I\right\vert $ to the coordinate $u_{I}$ and the weight $-\left\vert
I\right\vert $ to $\frac{\partial}{\partial u_{I}}$. (Note that this is the
convention made in \cite{HM}, and is \textit{different} from that made in
\cite{F1} and \cite{RS}: in the last two papers, the authors assign positive
weight also to derivatives). In the following we will say that a $C^{\infty}$
function $f$ has weight $\geqslant s$ if the Taylor expansion of $f$ at the
origin does not include terms of the kind $au_{I_{1}}u_{I_{2}}\cdots u_{I_{k}%
}$ with $a\neq0$ and $\left\vert I_{1}\right\vert +\left\vert I_{2}\right\vert
+\ldots+\left\vert I_{k}\right\vert <s$. A vector field $Y=\sum_{I\in B}%
f_{I}e_{I}$ has weight $~\geqslant s$ if $f_{I}$ has weight $~\geqslant
s+\left\vert I\right\vert $ for every $I\in B$.
\end{definition}

Note that the weight of a function is always $\geq0$, while the weight of a
vector field is $\geq-r,$ where $r$ is as above.

We want to stress that the definition of weight relies on the canonical
coordinates, therefore it depends on the choice of a particular basis $B$ of
$\mathbb{R}^{p}.$

In the following we shall denote with $F_{s}^{q}$ the set of functions such
that in their Taylor expansion of degree $\leqslant q$ (in the standard
sense), all terms have weight $\geq s$. Also $V_{s}^{q}$ will denote the set
of the vector fields with a similar property. The subset of $F_{s}^{q}$ and
$V_{s}^{q}$ of elements that vanish at $u=0$ will be denoted by $\mathring
{F}_{s}^{q}$ and $\mathring{V}_{s}^{q}.$

\begin{lemma}
\label{Inclusions}The following inclusions hold
\[%
\begin{array}
[c]{ccc}%
F_{s}^{q}F_{t}^{q}\subset F_{s+t}^{q} & \mathring{F}_{s}^{q}F_{t}^{q-1}%
\subset\mathring{F}_{s+t}^{q} & \\
F_{s}^{q}V_{t}^{q}\subset V_{s+t}^{q} & \mathring{F}_{s}^{q}V_{t}^{q-1}%
\subset\mathring{V}_{s+t}^{q} & F_{s}^{q-1}\mathring{V}_{t}^{q}\subset
\mathring{V}_{s+t}^{q}\\
V_{s}^{q-1}\left(  F_{s}^{q}\right)  \subset V_{s+t}^{q-1} & \mathring{V}%
_{s}^{q}\left(  F_{t}^{q}\right)  \subset\mathring{F}_{s+t}^{q} & \\
\left[  V_{s}^{q},V_{t}^{q}\right]  \subset V_{s+t}^{q-1} & \left[
\mathring{V}_{s}^{q},V_{t}^{q-1}\right]  \subset V_{s+t}^{q-1} &
\end{array}
\]
with the obvious meaning of the symbols.
\end{lemma}

\begin{proof}
If $f\in F_{s}^{q}$ and $g\in F_{t}^{q}$, then all terms of the product of
their Taylor expansion of degree $\leqslant q$ have weight $\geq s+t$.
Therefore, the same is true for the Taylor expansion of degree $\leqslant q$
of $fg$, so that $fg\in F_{s+t}^{q}$. This shows that the first inclusion
holds, and the second one is an immediate corollary. The other inclusions can
be proved by means of similar arguments.
\end{proof}

For any vector field $X_{\left[  J\right]  }$ with $\left\vert J\right\vert
\leq s,$ we can express $X_{\left[  J\right]  }$ in terms of the basis
$\left\{  X_{\left[  I\right]  }\right\}  _{I\in B},$ writing%
\[
X_{\left[  J\right]  }=\sum_{I\in B}c_{JI}\left(  u\right)  X_{\left[
I\right]  }%
\]
for suitable functions $c_{JI}.$ \hfill

\begin{theorem}
\label{Thm weight of X_[I]}For every multiindices $I$ the vector field
$X_{\left[  I\right]  }$ has weight $\geqslant-\left\vert I\right\vert .$
\end{theorem}

\begin{proof}
Throughout the proof we will assume the $X_{\left[  I\right]  }$ written in
canonical coordinates. In particular, the point $\overline{x}$ corresponds to
$u=0.$ We shall prove by induction on $q\geqslant0$ the following two facts:

\begin{itemize}
\item[i)] For every multiindex $I$ we have $X_{\left[  I\right]  }\in
V_{-\left\vert I\right\vert }^{q}$;

\item[ii)] For every positive integer $\alpha$, if for every $m\leqslant q+1$
and for every multiindices $I_{1},\ldots I_{m}\in B\ $such that $\left\vert
I_{1}\right\vert +\ldots\left\vert I_{m}\right\vert <\alpha$ we have
$X_{\left[  I_{m}\right]  }\cdots X_{\left[  I_{1}\right]  }f\left(  0\right)
=0$, then $f\in F_{\alpha}^{q+1}$.
\end{itemize}

First of all we observe that it is enough to show that $X_{\left[  I\right]
}\in V_{-\left\vert I\right\vert }^{q}$ when $I\in B$. Indeed let us fix a
certain $q\geqslant0$ and assume to know that $X_{\left[  I\right]  }\in
V_{-\left\vert I\right\vert }^{q}$ for every $I\in B$ and that ii) holds. Let
$J$ be a multiindex, $J\notin B$. Then%
\[
X_{\left[  J\right]  }=\sum_{I\in B}c_{JI}X_{\left[  I\right]  }.
\]
Since the vector fields $X_{0},X_{1},...,X_{n}$ are free up to step $r$, we
can assume that in the last sum $c_{JI}$ is nonzero only if $\left\vert
I\right\vert =\left\vert J\right\vert .$ For these constants $c_{JI}$ we then
have%
\[
X_{\left[  J\right]  }=\sum_{I\in B,\left\vert I\right\vert =\left\vert
J\right\vert }c_{JI}X_{\left[  I\right]  },
\]
which shows that it is enough to prove i) for $I\in B.$

Let now $q=0$. Observe that composing the operator $\operatorname*{ad}e_{I}$
with (\ref{Euler}) we get%
\begin{equation}
X_{\left[  I\right]  }+\sum_{K\in B}u_{K}\operatorname*{ad}e_{I}\left(
X_{\left[  K\right]  }\right)  =e_{I} \label{aaa}%
\end{equation}
This implies $\left(  X_{\left[  I\right]  }\right)  _{0}=e_{I}$ and therefore
that $X_{\left[  I\right]  }\in V_{-\left\vert I\right\vert }^{0}$. Assume now
that $f\left(  0\right)  =0$ and that for every multiindex $I\in B,$
$\left\vert I\right\vert <\alpha$ we have $X_{\left[  I\right]  }f\left(
0\right)  =0$. Since $X_{\left[  I\right]  }f\left(  0\right)  =\frac{\partial
f}{\partial u_{I}}\left(  0\right)  $ we have $f\in F_{\alpha}^{1}$.

Assume now that i) and ii) hold for a certain $q$ and let us prove that the
same is true with $q$ replaced by $q+1$. We start with i). We claim that it is
enough to show that
\[
W=\operatorname*{ad}X_{\left[  J\right]  }\left(  e_{I}\right)  \in
V_{-\left\vert I\right\vert -\left\vert J\right\vert }^{q}%
\]
Indeed, $u_{K}\operatorname*{ad}e_{I}\left(  X_{\left[  K\right]  }\right)
\in V_{-\left\vert I\right\vert }^{q+1}$ by Lemma \ref{Inclusions}, and since
$e_{I}\in V_{-\left\vert I\right\vert }^{q+1}$ by (\ref{aaa}) we have
$X_{\left[  I\right]  }\in V_{-\left\vert I\right\vert }^{q+1}$.

In order to show that $W\in V_{-\left\vert I\right\vert -\left\vert
J\right\vert }^{q}$ we compose $\operatorname*{ad}X_{\left[  J\right]  }$ with
(\ref{aaa}). This yields%
\begin{align}
\operatorname*{ad}X_{\left[  J\right]  }e_{I}  &  =W=\operatorname*{ad}%
X_{\left[  J\right]  }X_{\left[  I\right]  }+\sum_{K\in B}u_{K}%
\operatorname*{ad}X_{\left[  J\right]  }\operatorname*{ad}e_{I}\left(
X_{\left[  K\right]  }\right)  +\label{bbb}\\
&  +\sum_{K\in B}X_{\left[  J\right]  }\left(  u_{K}\right)
\operatorname*{ad}e_{I}\left(  X_{\left[  K\right]  }\right)  .\nonumber
\end{align}
From (\ref{aaa}) we also get $X_{\left[  J\right]  }\left(  u_{K}\right)
=\delta_{JK}-\sum_{L\in B}u_{L}\operatorname*{ad}e_{J}\left(  X_{\left[
L\right]  }\right)  \left(  u_{K}\right)  $. Since $\operatorname*{ad}%
e_{J}\left(  X_{\left[  L\right]  }\right)  \in V_{-\left\vert J\right\vert
-\left\vert L\right\vert }^{q-1}$ we have $u_{L}\operatorname*{ad}e_{J}\left(
X_{\left[  L\right]  }\right)  \in\mathring{V}_{-\left\vert J\right\vert }%
^{q}$ and therefore
\[
u_{L}\operatorname*{ad}e_{J}\left(  X_{\left[  L\right]  }\right)  \left(
u_{K}\right)  \in\mathring{F}_{-\left\vert J\right\vert +\left\vert
K\right\vert }^{q}.
\]
This implies that the second summation in (\ref{bbb}) is congruent to
$\operatorname*{ad}e_{I}\left(  X_{\left[  J\right]  }\right)  =-W$ modulo
$\mathring{V}_{-\left\vert J\right\vert -\left\vert I\right\vert }^{q}$. By
Jacobi identity (see the proof of Proposition \ref{Proposition 2})
\[
\operatorname*{ad}X_{\left[  J\right]  }X_{\left[  I\right]  }=\sum
_{\left\vert L\right\vert =\left\vert J\right\vert +\left\vert I\right\vert
}c_{L}X_{\left[  L\right]  }%
\]
for suitable coefficients $c_{L}$. This implies $\operatorname*{ad}X_{\left[
J\right]  }X_{\left[  I\right]  }\in V_{-\left\vert J\right\vert -\left\vert
I\right\vert }^{q}$. Hence%
\begin{equation}
W\equiv-W+\sum_{K\in B}u_{K}\operatorname*{ad}X_{\left[  J\right]
}\operatorname*{ad}e_{I}\left(  X_{\left[  K\right]  }\right)
\;\;\operatorname{mod}V_{-\left\vert J\right\vert -\left\vert I\right\vert
}^{q}. \label{ccc}%
\end{equation}
Since%
\begin{align*}
&  u_{K}\operatorname*{ad}X_{\left[  J\right]  }\operatorname*{ad}e_{I}\left(
X_{\left[  K\right]  }\right)  =-u_{K}\operatorname*{ad}X_{\left[  J\right]
}\operatorname*{ad}X_{\left[  K\right]  }\left(  e_{I}\right)  =\\
&  =-u_{K}\operatorname*{ad}X_{\left[  K\right]  }\operatorname*{ad}X_{\left[
J\right]  }\left(  e_{I}\right)  -u_{K}\operatorname*{ad}\left[  X_{\left[
K\right]  },X_{\left[  J\right]  }\right]  \left(  e_{I}\right)
\end{align*}
we have%
\[
u_{K}\operatorname*{ad}X_{\left[  J\right]  }\operatorname*{ad}e_{I}\left(
X_{\left[  K\right]  }\right)  \equiv-u_{K}\operatorname*{ad}X_{\left[
K\right]  }\operatorname*{ad}X_{\left[  J\right]  }\left(  e_{I}\right)
\;\;\operatorname{mod}V_{-\left\vert I\right\vert -\left\vert J\right\vert
}^{q}.
\]
By Jacoby identity and substituting in (\ref{ccc}) we obtain
\[
2W=-\sum_{K}u_{K}\operatorname*{ad}X_{\left[  K\right]  }\left(  W\right)
\;\;\operatorname{mod}V_{-\left\vert I\right\vert -\left\vert J\right\vert
}^{q}.
\]
We now use (\ref{Euler}) to replace $X_{\left[  K\right]  }$ by $e_{K}$.
Indeed we have%
\begin{align*}
\sum_{K}u_{K}\operatorname*{ad}X_{\left[  K\right]  }\left(  W\right)   &
=\sum_{K}\operatorname*{ad}\left(  u_{K}X_{\left[  K\right]  }\right)  \left(
W\right)  +\sum_{K}W\left(  u_{k}\right)  X_{\left[  K\right]  }\\
&  =\operatorname*{ad}\left(  \sum_{K}u_{K}X_{\left[  K\right]  }\right)
\left(  W\right)  +\sum_{K}W\left(  u_{k}\right)  X_{\left[  K\right]  }\\
&  =\operatorname*{ad}\left(  \sum_{K}u_{K}e_{K}\right)  \left(  W\right)
+\sum_{K}W\left(  u_{k}\right)  X_{\left[  K\right]  }\\
&  =\sum_{K}u_{K}\operatorname*{ad}\left(  e_{K}\right)  \left(  W\right)
+\sum_{K}W\left(  u_{K}\right)  \left(  X_{\left[  K\right]  }-e_{K}\right)  .
\end{align*}
Since $W=\operatorname*{ad}X_{\left[  J\right]  }e_{I}\in V_{-\left\vert
I\right\vert -\left\vert J\right\vert }^{q-1}$ we have $W\left(  u_{K}\right)
\in F_{\left\vert K\right\vert -\left\vert I\right\vert -\left\vert
J\right\vert }^{q-1}.$ Also, since $X_{\left[  K\right]  }-e_{K}\in
\mathring{V}_{-\left\vert K\right\vert }^{q}$ we have $W\left(  u_{K}\right)
\left(  X_{\left[  K\right]  }-e_{K}\right)  \in V_{-\left\vert I\right\vert
-\left\vert J\right\vert }^{q}$. Hence%
\[
TW\in V_{-\left\vert I\right\vert -\left\vert J\right\vert }^{q}%
\]
where we set $TW=2W+\sum_{K}u_{K}\operatorname*{ad}\left(  e_{K}\right)
\left(  W\right)  $. We claim that this implies that $W\in V_{-\left\vert
I\right\vert -\left\vert J\right\vert }^{q}$. To see this, write $W=\sum
_{L}f_{L}e_{L}$. Then
\begin{align*}
TW  &  =2\sum_{L}f_{L}e_{L}+\sum_{L,K}u_{K}\left[  e_{K},f_{L}e_{L}\right]
=\\
&  =2\sum_{L}f_{L}e_{L}+\sum_{L,K}u_{K}\frac{\partial f_{L}}{\partial u_{K}%
}e_{L}=\sum_{L}\left(  2f_{L}+\sum_{K}u_{K}\frac{\partial f}{\partial u_{K}%
}\right)  e_{L}.
\end{align*}
Let $g$ be a homogeneous function of degree $\mu$, then $\sum_{K}u_{K}%
\frac{\partial g}{\partial u_{K}}=\mu g$ which shows that the operator
$f\mapsto2f+\sum_{K}u_{K}\frac{\partial f}{\partial u_{K}}$ acts on the Taylor
expansion of a function multiplying a term of degree $\mu$ by $\left(
2+\mu\right)  $. This implies that $W\in V_{-\left\vert I\right\vert
-\left\vert J\right\vert }^{q}$.

Now we will show that also ii) holds with $q$ replaced by $q+1$. We have to
show that if for every $m\leqslant q+2$ and for every multiindices
$I_{1},\ldots I_{m}\in B$ such that $\left\vert I_{1}\right\vert
+\ldots\left\vert I_{m}\right\vert <\alpha$ we have $X_{\left[  I_{m}\right]
}\cdots X_{\left[  I_{1}\right]  }f\left(  0\right)  =0,$ then $f\in
F_{\alpha}^{q+2}$. By the definition of the class $F_{\alpha}^{q+2}$ this
amounts to showing that $e_{I_{m}}\cdots e_{I_{1}}f\left(  0\right)  =0$ for
every $m\leqslant q+2$ and every $I_{1},\ldots I_{m}\in B\ $such that
$\left\vert I_{1}\right\vert +\ldots\left\vert I_{m}\right\vert <\alpha$.

By the induction hypothesis $f\in F_{\alpha}^{q+1}$, we already know that for
every $I\in B_{\overline{x}}$ we have $X_{\left[  I\right]  }\in
V_{-\left\vert I\right\vert }^{q+1}$. By (\ref{aaa}) we have
\[
e_{I_{1}}f=X_{\left[  I_{1}\right]  }f+g_{1}%
\]
with $g_{1}=\sum_{K\in B}u_{K}\operatorname*{ad}e_{I_{1}}\left(  X_{\left[
K\right]  }\right)  f$. By Lemma \ref{Inclusions} we have $g_{1}\in
F_{\alpha-\left\vert I_{1}\right\vert }^{q+1}$. Iterating this argument yields%
\[
e_{I_{m}\cdots}e_{I_{1}}f=X_{\left[  I_{m}\right]  }\cdots X_{\left[
I_{1}\right]  }f+g_{m}%
\]
with $g_{m}\in F_{\alpha-\left(  \left\vert I_{1}\right\vert +\cdots
+\left\vert I_{m}\right\vert \right)  }^{q+2-m}$. Since $\alpha>\left\vert
I_{1}\right\vert +\cdots+\left\vert I_{m}\right\vert $ this implies
$g_{m}\left(  0\right)  =0$ and therefore that $e_{I_{m}\cdots}e_{I_{1}%
}f\left(  0\right)  =0$. \hfill
\end{proof}

\subsection{Pointwise approximation}

As in the previous subsection, we assume that the $X_{i}$'s are free and a
basis $\left\{  \left(  X_{\left[  I\right]  }\right)  _{\overline{x}%
}\right\}  _{I\in B}$ for $\mathbb{R}^{p}$ is chosen once and for all; this
choice induces a system of canonical coordinates $u\ $near $\overline{x}$ such
that%
\begin{equation}
\sum_{I\in B}u_{I}e_{I}=\sum_{I\in B}u_{I}X_{\left[  I\right]  }
\label{Euler2}%
\end{equation}
(where $e_{I}=\frac{\partial}{\partial u_{I}}$) and that for every multiindex
$I$ the vector field $X_{\left[  I\right]  }$ has weight $\geq-\left\vert
I\right\vert $.

We can now prove Rothschild-Stein's approximation theorem for free weighted
vector fields:

\begin{theorem}
[Approximation, pointwise version]\label{Thm Approximation}If $Y_{0}%
,Y_{1},\ldots,Y_{n}$ is another system of vector fields satisfying (with
respect to the same canonical coordinates)%
\begin{equation}
\sum_{I\in B}u_{I}e_{I}=\sum_{I\in B}u_{I}Y_{\left[  I\right]  }\ ,
\label{EulerY2}%
\end{equation}
then $X_{\left[  I\right]  }-Y_{\left[  I\right]  }$ has weight $\geq
1-\left\vert I\right\vert $.
\end{theorem}

(In particular, for $I=\left(  i\right)  $ we have that $X_{0}-Y_{0}$ has
weight $\geq-1$ while $X_{i}-Y_{i}$ has weight $\geq0$ for $i=1,2,...,n$).

\begin{proof}
The proof exploits the same techniques as in the proof of Theorem
\ref{Thm weight of X_[I]}. Let us recall that we are now working in canonical
coordinates. We shall prove by induction on $q$ that
\begin{equation}
X_{\left[  I\right]  }-Y_{\left[  I\right]  }\in V_{1-\left\vert I\right\vert
}^{q}. \label{Approx}%
\end{equation}
Observe that this is obvious when $\left\vert I\right\vert >s$. Indeed, every
vector field $Z=\sum_{J\in B}f_{J}e_{J}$ has a weight $\geqslant-s$ since if
$J\in B$ then $\left\vert J\right\vert \leqslant s$. Next we prove that it is
enough to show (\ref{Approx}) when $I\in B$.

Indeed, let $I$ be any multiindex of weight $\leqslant s$. Then since
$\left\{  \left(  X_{\left[  J\right]  }\right)  _{\overline{x}}\right\}
_{J\in B}$ spans $\mathbb{R}^{p}$ there exist coefficients $c_{IJ}$ such that%
\[
\left(  X_{\left[  I\right]  }\right)  _{\overline{x}}=\sum_{J\in B}%
c_{IJ}\left(  X_{\left[  J\right]  }\right)  _{\overline{x}}.
\]
Let $a_{J}=\delta_{IJ}-c_{IJ}$ where we assume that $c_{IJ}=0$ if $J\notin B$.
Then%
\[
\sum a_{J}\left(  X_{\left[  J\right]  }\right)  _{\overline{x}}=0.
\]
Since the vector fields are free of weight $s$ this implies that for every $K$%
\[
\sum_{J}a_{J}A_{JK}=0.
\]
Therefore, for any family of vector fields $Z_{i}$ we have%
\[
0=\sum_{K,~\left\vert K\right\vert =\left\vert I\right\vert }\sum_{J}%
a_{J}A_{JK}Z_{K}=\sum_{J,~\left\vert J\right\vert =\left\vert I\right\vert
}\sum_{K}a_{J}A_{JK}Z_{K}=\sum_{\left\vert J\right\vert =\left\vert
I\right\vert }a_{J}Z_{\left[  J\right]  }%
\]
and finally%
\[
Z_{\left[  I\right]  }=\sum_{J\in B,\left\vert J\right\vert =\left\vert
I\right\vert }c_{IJ}Z_{\left[  J\right]  }.
\]
In particular we have%
\[
X_{\left[  I\right]  }=\sum_{J\in B,\left\vert J\right\vert =\left\vert
I\right\vert }c_{IJ}X_{\left[  J\right]  }%
\]
and%
\[
Y_{\left[  I\right]  }=\sum_{J\in B,\left\vert J\right\vert =\left\vert
I\right\vert }c_{IJ}Y_{\left[  J\right]  }.
\]
These identities show that if (\ref{Approx}) hold for $I\in B,$ then they hold
for every $I.$

Now we prove that (\ref{Approx}) hold for $I\in B$. Let $q=0$. We have seen in
the proof of Theorem \ref{Thm weight of X_[I]} that (\ref{EulerY2}) implies%
\begin{equation}
X_{\left[  I\right]  }+\sum_{K\in B}u_{K}\operatorname*{ad}e_{I}\left(
X_{\left[  K\right]  }\right)  =e_{I} \label{aaa2}%
\end{equation}
and%
\begin{equation}
Y_{\left[  I\right]  }+\sum_{K\in B}u_{K}\operatorname*{ad}e_{I}\left(
Y_{\left[  K\right]  }\right)  =e_{I}. \label{aaaY2}%
\end{equation}
In particular, $\left(  X_{\left[  I\right]  }\right)  _{\overline{x}}=e_{I}$
and $\left(  Y_{\left[  I\right]  }\right)  _{\overline{x}}=e_{I}$, hence
(\ref{Approx}) holds for $q=0$.

We now assume that (\ref{Approx}) holds for a certain $q$ and we prove that
the same is true with $q$ replaced by $q+1$. We claim that it is enough to
show that
\[
Z=\operatorname*{ad}\left(  X_{\left[  J\right]  }-Y_{\left[  J\right]
}\right)  \left(  e_{I}\right)  \in V_{1-\left\vert I\right\vert -\left\vert
J\right\vert }^{q}.
\]
Indeed, by (\ref{aaa2}), (\ref{aaaY2})%
\[
X_{\left[  I\right]  }-Y_{\left[  I\right]  }=\sum_{K\in B}u_{K}%
\operatorname*{ad}\left(  X_{\left[  K\right]  }-Y_{\left[  K\right]
}\right)  \left(  e_{I}\right)
\]
and by Lemma \ref{Inclusions}, $u_{K}\operatorname*{ad}\left(  X_{\left[
K\right]  }-Y_{\left[  K\right]  }\right)  \left(  e_{I}\right)  \in
V_{1-\left\vert I\right\vert }^{q+1}$.

Now we will show that $Z=\operatorname*{ad}\left(  X_{\left[  J\right]
}-Y_{\left[  J\right]  }\right)  \left(  e_{I}\right)  \in V_{1-\left\vert
I\right\vert -\left\vert J\right\vert }^{q}$. Composing (\ref{aaa2}) with
$\operatorname*{ad}X_{\left[  J\right]  }$, (\ref{aaaY2}) with
$\operatorname*{ad}Y_{\left[  J\right]  }$ and computing the difference gives%
\begin{align*}
Z  &  =\operatorname*{ad}X_{\left[  J\right]  }\left(  X_{\left[  I\right]
}\right)  -\operatorname*{ad}Y_{\left[  J\right]  }\left(  Y_{\left[
I\right]  }\right)  +\\
&  +\sum_{K\in B}\operatorname*{ad}X_{\left[  J\right]  }\left(
u_{K}\operatorname*{ad}e_{I}\left(  X_{\left[  K\right]  }\right)  \right)
-\sum_{K\in B}\operatorname*{ad}Y_{\left[  J\right]  }\left(  u_{K}%
\operatorname*{ad}e_{I}\left(  Y_{\left[  K\right]  }\right)  \right) \\
&  =\operatorname*{ad}X_{\left[  J\right]  }\left(  X_{\left[  I\right]
}\right)  -\operatorname*{ad}Y_{\left[  J\right]  }\left(  Y_{\left[
I\right]  }\right)  +\sum_{K\in B}\operatorname*{ad}\left(  X_{\left[
J\right]  }-Y_{\left[  J\right]  }\right)  \left(  u_{K}\operatorname*{ad}%
e_{I}\left(  X_{\left[  K\right]  }\right)  \right)  +\\
&  +\sum_{K\in B}\operatorname*{ad}Y_{\left[  J\right]  }\left(
u_{K}\operatorname*{ad}e_{I}\left(  X_{\left[  K\right]  }-Y_{\left[
K\right]  }\right)  \right)
\end{align*}

By the Jacobi identity
\[
\operatorname*{ad}X_{\left[  J\right]  }\left(  X_{\left[  I\right]  }\right)
-\operatorname*{ad}Y_{\left[  J\right]  }\left(  Y_{\left[  I\right]
}\right)  =\sum_{\left\vert K\right\vert =\left\vert J\right\vert +\left\vert
I\right\vert }c_{K}\left(  X_{\left[  K\right]  }-Y_{\left[  K\right]
}\right)
\]
and by inductive hypothesis $\operatorname*{ad}X_{\left[  J\right]  }\left(
X_{\left[  I\right]  }\right)  -\operatorname*{ad}Y_{\left[  J\right]
}\left(  Y_{\left[  I\right]  }\right)  \in V_{1-\left\vert I\right\vert
-\left\vert J\right\vert }^{q}$.

Also since $X_{\left[  K\right]  }\in V_{-\left\vert K\right\vert }^{q+1}$ we
have $\operatorname*{ad}e_{I}\left(  X_{\left[  K\right]  }\right)  \in
V_{-\left\vert I\right\vert -\left\vert K\right\vert }^{q}$ and therefore
$u_{K}\operatorname*{ad}e_{I}\left(  X_{\left[  K\right]  }\right)
\in\mathring{V}_{-\left\vert I\right\vert }^{q+1}$. Since $X_{\left[
J\right]  }-Y_{\left[  J\right]  }\in V_{1-\left\vert J\right\vert }^{q}$ we
have
\[
\operatorname*{ad}\left(  X_{\left[  J\right]  }-Y_{\left[  J\right]
}\right)  \left(  u_{K}\operatorname*{ad}e_{I}\left(  X_{\left[  K\right]
}\right)  \right)  \in V_{1-\left\vert I\right\vert -\left\vert J\right\vert
}^{q}.
\]
This means that modulo $V_{1-\left\vert I\right\vert -\left\vert J\right\vert
}^{q}$ we have
\begin{align*}
Z  &  \equiv\sum_{K\in B}\operatorname*{ad}Y_{\left[  J\right]  }\left(
u_{K}\operatorname*{ad}e_{I}\left(  X_{\left[  K\right]  }-Y_{\left[
K\right]  }\right)  \right) \\
&  \equiv\sum_{K\in B}\operatorname*{ad}\left(  Y_{\left[  J\right]  }%
-e_{J}\right)  \left(  u_{K}\operatorname*{ad}e_{I}\left(  X_{\left[
K\right]  }-Y_{\left[  K\right]  }\right)  \right)  +\\
&  +\sum_{K\in B}\operatorname*{ad}e_{J}\left(  u_{K}\operatorname*{ad}%
e_{I}\left(  X_{\left[  K\right]  }-Y_{\left[  K\right]  }\right)  \right) \\
&  \equiv\sum_{K\in B}\operatorname*{ad}e_{J}\left(  u_{K}\operatorname*{ad}%
e_{I}\left(  X_{\left[  K\right]  }-Y_{\left[  K\right]  }\right)  \right)
\end{align*}
since $Y_{\left[  J\right]  }-e_{J}\in\mathring{V}_{-\left\vert J\right\vert
}^{q+1}$ and $u_{K}\operatorname*{ad}e_{I}\left(  X_{\left[  K\right]
}-Y_{\left[  K\right]  }\right)  \in V_{1-\left\vert I\right\vert }^{q}$.
Finally%
\begin{align}
Z  &  \equiv\sum_{K\in B}\operatorname*{ad}e_{J}\left(  u_{K}%
\operatorname*{ad}e_{I}\left(  X_{\left[  K\right]  }-Y_{\left[  K\right]
}\right)  \right) \label{ddd}\\
&  \equiv\sum_{K\in B}\left[  \operatorname*{ad}e_{J}\operatorname*{ad}%
e_{I}\left[  u_{K}\left(  X_{\left[  K\right]  }-Y_{\left[  K\right]
}\right)  \right]  -\operatorname*{ad}e_{J}\delta_{IK}\left(  X_{\left[
K\right]  }-Y_{\left[  K\right]  }\right)  \right] \nonumber\\
&  \equiv\operatorname*{ad}e_{J}\operatorname*{ad}e_{I}\left[  \sum_{K\in
B}u_{K}\left(  X_{\left[  K\right]  }-Y_{\left[  K\right]  }\right)  \right]
-\operatorname*{ad}e_{J}\left(  X_{\left[  I\right]  }-Y_{\left[  I\right]
}\right) \nonumber\\
&  \equiv-\operatorname*{ad}e_{J}\left(  X_{\left[  I\right]  }-Y_{\left[
I\right]  }\right)  \equiv\operatorname*{ad}\left(  X_{\left[  I\right]
}-Y_{\left[  I\right]  }\right)  \left(  e_{J}\right)  .\nonumber
\end{align}
since $\sum_{K\in B}u_{K}\left(  X_{\left[  K\right]  }-Y_{\left[  K\right]
}\right)  =0$ by (\ref{Euler2}) and \ref{EulerY2}). This shows that for every
multiindex $I$ and $J$ we have $\operatorname*{ad}\left(  X_{\left[  J\right]
}-Y_{\left[  J\right]  }\right)  \left(  e_{I}\right)  \equiv
\operatorname*{ad}\left(  X_{\left[  I\right]  }-Y_{\left[  I\right]
}\right)  \left(  e_{J}\right)  $. Using this fact in (\ref{ddd}) yields
\begin{align*}
Z  &  \equiv\sum_{K\in B}\operatorname*{ad}e_{J}\left(  u_{K}%
\operatorname*{ad}e_{I}\left(  X_{\left[  K\right]  }-Y_{\left[  K\right]
}\right)  \right) \\
&  \equiv\sum_{K\in B}\operatorname*{ad}e_{J}\left(  u_{K}\operatorname*{ad}%
e_{K}\left(  X_{\left[  I\right]  }-Y_{\left[  I\right]  }\right)  \right) \\
&  \equiv\operatorname*{ad}e_{J}\left(  X_{\left[  I\right]  }-Y_{\left[
I\right]  }\right)  +\sum_{K\in B}u_{K}\operatorname*{ad}e_{J}%
\operatorname*{ad}e_{K}\left(  X_{\left[  I\right]  }-Y_{\left[  I\right]
}\right) \\
&  \equiv\operatorname*{ad}e_{J}\left(  X_{\left[  I\right]  }-Y_{\left[
I\right]  }\right)  +\sum_{K\in B}u_{K}\operatorname*{ad}e_{K}%
\operatorname*{ad}e_{J}\left(  X_{\left[  I\right]  }-Y_{\left[  I\right]
}\right)
\end{align*}
since $e_{K}$ and $e_{J}$ commutes. Hence%
\[
Z\equiv-Z-\sum_{K\in B}u_{K}\operatorname*{ad}e_{K}Z.
\]
This means
\[
TZ\equiv0\;\;\operatorname{mod}V_{1-\left\vert I\right\vert -\left\vert
J\right\vert }^{q}.
\]
which implies $Z\in V_{1-\left\vert I\right\vert -\left\vert J\right\vert
}^{q}$. \medskip\hfill
\end{proof}

In order to recover from Theorem \ref{Thm Approximation} the exact statement
of Rothschild-Stein's \textquotedblleft approximation
theorem\textquotedblright, some work has still to be done. First, we have to
pass from the \textit{pointwise} statement of Theorem \ref{Thm Approximation}
to an analogous \textit{local} statement. This involves the introduction of
Rothschild-Stein's \textquotedblleft map $\Theta$\textquotedblright\ and the
study of some of its properties. Second, we have to apply this theorem to the
case where the vector fields $Y_{i}$ are homogeneous left invariant with
respect to a structure of homogeneous group, and deduce some information on
the \textquotedblleft remainders\textquotedblright\ in this approximation
procedure. These tasks will be performed in the next two subsections, respectively.

\subsection{From \textit{pointwise }to \textit{local}. The map $\Theta$}

We now revise the construction of local coordinates $u_{I}$ made in
\S \ref{subsection canonical coordinates}. Let $\Omega$ be as at the beginning
of \S \ref{section approximation}; we claim that for any $\Omega^{\prime
}\Subset\Omega$ there exists a neighborhood $U\left(  0\right)  \subset$
$\mathbb{R}^{p}$ where the map%
\begin{equation}
E\left(  \cdot,\xi_{0}\right)  :u\equiv\left(  u_{I}\right)  _{I\in
B}\longmapsto\xi\equiv\exp\left(
{\displaystyle\sum\limits_{I\in B}}
u_{I}{X}_{\left[  I\right]  }\right)  \left(  \xi_{0}\right)  \label{mappa E}%
\end{equation}
is well defined and smooth, for any fixed $\xi_{0}\in\Omega^{\prime}$. Namely,
by classical results about O.D.E.'s, $E$ is smooth in the joint variables
$\left(  u,\xi_{0}\right)  \in U\left(  0\right)  \times\Omega^{\prime}$.

Next, we define%
\[
F\left(  u,\xi_{0},\xi\right)  =E\left(  u,\xi_{0}\right)  -\xi
\]
on $U\left(  0\right)  \times\Omega^{\prime}\times\mathbb{R}^{p}.$ Noting that
$F\left(  0,\xi_{0},\xi_{0}\right)  =0$ and that the Jacobian of $F$ with
respect to the $u$ variables, at $\left(  0,\xi_{0},\xi_{0}\right)  ,$ has
determinant%
\[
\det\left(  \left(  {X}_{\left[  I\right]  }\right)  _{\xi_{0}}\right)  ,
\]
which does not vanish since $\left\{  {X}_{\left[  I\right]  }\right\}  _{I\in
B}$ span $\mathbb{R}^{p}$, by the implicit function theorem we can define a
function%
\[
u=\Theta\left(  \eta,\xi\right)  ,
\]
smooth in some neighborhood $W$ of $\left(  \xi_{0},\xi_{0}\right)  $, such
that $E\left(  \Theta\left(  \eta,\xi\right)  ,\eta\right)  =\xi.$

Summarizing the above discussion we can state the following:

\begin{proposition}
[The map $\Theta$]\label{Map Theta}

\begin{enumerate}
\item[i)] For any $\xi_{0}\in\Omega$ there exist a neighborhood $W$ of
$\left(  \xi_{0},\xi_{0}\right)  $ in $\mathbb{R}^{2p},$ a neighborhood
$U\left(  0\right)  $ of $0$ in $\mathbb{R}^{p}$ and a smooth map
$\Theta\left(  \cdot,\cdot\right)  :W\rightarrow U\left(  0\right)  $ such
that:%
\begin{equation}
\xi=\exp\left(
{\displaystyle\sum\limits_{I\in B}}
u_{I}{X}_{\left[  I\right]  }\right)  \left(  \eta\right)  \text{ for
}u=\Theta\left(  \eta,\xi\right)  ; \label{def Theta}%
\end{equation}

\item[ii)] The map $\Theta$ satisfies
\begin{equation}
\Theta\left(  \eta,\xi\right)  =-\Theta\left(  \xi,\eta\right)  ;
\label{meno Theta}%
\end{equation}

\item[iii)] for any fixed $\eta,$ the map $u=\Theta\left(  \eta,\xi\right)  $
is a diffeomorphism from a neighborhood of $\eta$ onto a neighborhood of $0$,
in $\mathbb{R}^{p};$

\item[iv)] analogously, for any fixed $\xi,$ the map $u=\Theta\left(  \eta
,\xi\right)  $ is a diffeomorphism from a neighborhood of $\xi$ onto a
neighborhood of $0.$
\end{enumerate}
\end{proposition}

\begin{proof}
We have already proved (i) and (iii); (ii) follows form the fact that, for any
vector field $X$,%
\[
\xi=\exp\left(  X\right)  \left(  \eta\right)  \Longrightarrow\eta=\exp\left(
-X\right)  \left(  \xi\right)  ,
\]
as can be checked by definition of the exponential map; (iv) is then a
consequence of (ii) and (iii). \medskip\hfill
\end{proof}

The map $\Theta$ allows one to restate Theorem \ref{Thm Approximation}
(approximation) in a form more similar to that of Rothschild-Stein.

Recall that a vector field $Z$ has weight $k$ at some fixed point $\eta$ if
$Z$, expressed in terms of the local coordinates $u=\Theta\left(  \eta
,\xi\right)  ,$ has weight $k$ at $u=0,$ in the sense of Definition
\ref{Definition weights}.

It will be useful to recall also the concrete meaning of expressing the same
vector field in different coordinates: if we denote by $Z^{\xi}$ and $Z^{u},$
the vector field $Z$ written as a differential operator which acts on the
variables $\xi$ or $u$, respectively, then
\begin{equation}
Z^{\xi}\left[  f\left(  \Theta\left(  \eta,\xi\right)  \right)  \right]
=\left(  Z^{u}f\right)  \left(  \Theta\left(  \eta,\xi\right)  \right)  .
\label{Zeta csi - u}%
\end{equation}
for any smooth function $f\left(  u\right)  $.

Then we have:

\begin{theorem}
[Approximation, local version]\label{Thm Approximation local}For every
multiindex $I$ the vector field ${X}_{\left[  I\right]  }$ has weight
$\geq-\left\vert I\right\vert $ at any point of $\Omega$. If $Y_{0}%
,Y_{1},\ldots,Y_{n}$ is another system of vector fields (expressed in the same
coordinates $u$) satisfying%
\begin{equation}
\sum_{I\in B}u_{I}e_{I}=\sum_{I\in B}u_{I}Y_{\left[  I\right]  }
\label{Euler Y}%
\end{equation}
then ${X}_{\left[  I\right]  }-Y_{\left[  I\right]  }$ has weight
$\geq1-\left\vert I\right\vert $ at any point of $\Omega$. Moreover, for any
point $\eta\in\Omega$ there exists a system of vector fields $R_{\eta,\left[
I\right]  },$ of weight $\geq1-\left\vert I\right\vert $ at $\eta$ (when
expressed in the coordinates $u$) and smoothly depending on the point $\eta,$
such that%
\begin{equation}
{X}_{\left[  I\right]  }^{\xi}\left[  f\left(  \Theta\left(  \eta,\xi\right)
\right)  \right]  =\left(  Y_{\left[  I\right]  }f\right)  \left(
\Theta\left(  \eta,\xi\right)  \right)  +\left(  R_{\eta,\left[  I\right]
}f\right)  \left(  \Theta\left(  \eta,\xi\right)  \right)  .
\label{approximation}%
\end{equation}

\end{theorem}

\begin{proof}
The first part of the theorem is exactly Theorem \ref{Thm weight of X_[I]}
plus Theorem \ref{Thm Approximation}, stated at any point $\eta.$ Saying that
${X}_{\left[  I\right]  }-Y_{\left[  I\right]  }$ has weight $\geq1-\left\vert
I\right\vert $ at $\eta,$ just by definition means that the vector field
$R_{\eta,\left[  I\right]  }={X}_{\left[  I\right]  }^{u}-Y_{\left[  I\right]
}$ has weight $\geq1-\left\vert I\right\vert $ at $u=0$. Here the superscript
$u$ in ${X}_{\left[  I\right]  }^{u}$ emphasizes that this vector field is
expressed in terms of the coordinates $u.$ By (\ref{Zeta csi - u}), we can
rewrite it in terms of coordinates $\xi,$ getting (\ref{approximation}). It
remains to check that $R_{\eta,\left[  I\right]  }$ depends smoothly on
$\eta.$ Let%
\[
R_{\eta,\left[  I\right]  }=\sum_{J}b_{IJ}\left(  \eta,u\right)
\partial_{u_{J}};
\]
then, applying (\ref{approximation}) to the function $f\left(  u\right)
=u_{J}$ we get%
\[
b_{IJ}\left(  \eta,\Theta\left(  \eta,\xi\right)  \right)  = {X}_{\left[
I\right]  }^{\xi}\left[  \left(  \Theta\left(  \eta,\xi\right)  \right)
_{J}\right]  -\left(  Y_{\left[  I\right]  }u_{J}\right)  \left(
\Theta\left(  \eta,\xi\right)  \right)  .
\]
The right-hand side of this equation is a smooth function of $\left(
\eta\,,\xi\right)  ,$ since $\Theta$ is smooth (see Proposition
\ref{Map Theta}); hence the functions%
\[
\left(  \eta\,,\xi\right)  \longmapsto b_{IJ}\left(  \eta,\Theta\left(
\eta,\xi\right)  \right)
\]
are smooth; fixing $\xi$ and composing with the diffeomorphism $u=\Theta
\left(  \eta,\xi\right)  $ we read that $b_{IJ}\left(  \eta,u\right)  $ are
smooth functions, which is what we needed to prove. \hfill
\end{proof}

\begin{remark}
The last statement is perfectly analogous of the approximation theorem proved
by Rothschild-Stein, but somewhat more general, since the vector fields
$Y_{\left[  I\right]  }$ need not be left invariant on a homogeneous group;
they only need to satisfy (\ref{Euler Y}).
\end{remark}

\subsection{Approximation by left invariant vector fields}

The standard application of Theorem \ref{Thm Approximation local} requires the
construction of a particular system of vector fields $\left\{  Y_{\left[
I\right]  }\right\}  _{I\in B}$ enjoying special properties.

In the following statement, the numbers $p,n,r$ keep the same meaning they
have in the previous subsections; also the system of multiindices $\left(
I\in B\right)  $ is the same.

\begin{theorem}
\label{ilteoremabrutto}There exist in $\mathbb{R}^{p}$ a system of smooth
vector fields $Y_{0},Y_{1},...,Y_{n}$ and a structure of homogeneous group
$G$, that is, a Lie group operation $u\circ v$ (\textquotedblleft
translation\textquotedblright) and a one-parameter family $\left\{
\delta_{\lambda}\right\}  _{\lambda>0}$ of automorphisms (\textquotedblleft
dilations\textquotedblright), acting as%
\[
\delta_{\lambda}\left(  \left(  u_{I}\right)  _{I\in B}\right)  =\left(
\lambda^{\left\vert I\right\vert }u_{I}\right)  _{I\in B}\text{ ,}%
\]
such that:

\begin{enumerate}
\item[(i)] the vector fields $Y_{0},Y_{1},...,Y_{n}$ are free up to weight $r$
in $\mathbb{R}^{p}$ and the vectors $\left\{  \left(  Y_{\left[  I\right]
}\right)  _{u}\right\}  _{\left\vert I\right\vert \leq r}$ span $\mathbb{R}%
^{p}$ at any point $u$ of the space;

\item[(ii)] the $Y_{\left[  I\right]  }$'s are left invariant and homogeneous
of degree $\left\vert I\right\vert $ with respect to\ the dilations in $G$;

\item[(iii)] at $u=0$, the $Y_{\left[  I\right]  }$'s coincide with the local
basis associated to the coordinates $u_{I}$, that is,
\[
\left(  Y_{\left[  I\right]  }\right)  _{0}=\frac{\partial}{\partial u_{I}};
\]

\item[(iv)] the $Y_{\left[  I\right]  }$'s satisfy (\ref{Euler Y});

\item[(v)] in the group $G$, the inverse $u^{-1}$ of an element is just its
(Euclidean) opposite $-u$.
\end{enumerate}

We stress that all the previous properties hold simultaneously, with respect
to the same system of coordinates in the space $\mathbb{R}^{p}.$ These
coordinates will be identified with the canonical coordinates $u$ induced by
the vector fields ${X}_{\left[  I\right]  }$ (see
\S \ref{subsection canonical coordinates}).
\end{theorem}

\begin{proof}
For the following abstract construction we refer to \cite[pp.3-15]{Ric}.

1. Let us consider the Lie algebra $\mathfrak{g}$ obtained by quotienting the
free Lie algebra with generators $Z_{0},...,Z_{n}$ with respect to the ideal
spanned by all commutators of weight greater than $r$ (this is called the
\textit{free nilpotent Lie algebra of type II\/} in \cite{RS}); here
$Z_{0},...,Z_{n}$ are thought as abstract generators, having weight $2$
$\left(  Z_{0}\right)  $ and $1$ $\left(  Z_{1},Z_{2},...,Z_{n}\right)  .$
This abstract Lie algebra is isomorphic to $\mathbb{R}^{d}$ for some $d$. We
claim that actually $d=p$. Namely the structure of the free Lie algebra of
type II of step $r$ on $n$ generators can depend only on $n,r,$ and since the
Lie algebra generated by the $X_{i}$'s is $\mathbb{R}^{p},$ $p=d.$

Hence the Lie algebra $\mathfrak{g}$ will be identified with $\mathbb{R}^{p}$
from now on.

2. We then introduce in $\mathbb{R}^{p}$ an operation $\circ,$ defined by:%
\begin{equation}
x\circ y\equiv S\left(  x,y\right)  \equiv x+y+\frac{1}{2}\left[  x,y\right]
+\frac{1}{12}\left[  \left[  x,y\right]  ,y\right]  -\frac{1}{12}\left[
\left[  x,y\right]  ,x\right]  +... \label{S(x,y)}%
\end{equation}
In the previous formula, $\left[  x,y\right]  $ denotes the commutator in the
Lie algebra $\mathfrak{g}$ (whose elements have been identified with points of
$\mathbb{R}^{p}$); the sum is finite because the Lie algebra is nilpotent, and
the precise definition of $S$ is given by
\[
S\left(  \cdot,\cdot\right)  =S^{\prime}\left(  1,1,\cdot,\cdot\right)
\]
where $S^{\prime}$ is the function appearing in the Baker-Campbell-Hausdorff
formula:%
\begin{equation}
\exp\left(  sX\right)  \exp\left(  tY\right)  =\exp\left(  S^{\prime}\left(
s,t,X,Y\right)  \right)  \label{BCH}%
\end{equation}
which holds, for $s,t$ small enough, for any couple of smooth vector fields
$X,Y$ which generate a finite dimensional Lie algebra. More precisely, it is
known that%
\begin{equation}
S\left(  x,y\right)  =\sum_{j+k\geq1}Z_{j,k}\left(  x,y\right)
\label{S(x,y) Z}%
\end{equation}
where each $Z_{j,k}\left(  x,y\right)  $ is a fixed linear combination of
iterated commutators of $x$ and $y,$ containing $j$ times $x$ and $k$ times
$y$. In terms of coordinates in $\mathbb{R}^{p},$ this function can be written
as%
\[
S\left(  x,y\right)  =\left(  S_{1}\left(  x,y\right)  ,S_{2}\left(
x,y\right)  ,...,S_{p}\left(  x,y\right)  \right)
\]
where each $S_{j}$ is a polynomial in $x,y$. Then (see Theorem 4.2 in
\cite{Ric}) the operation $\circ$ defines in $\mathbb{R}^{p}$ a structure of
homogeneous Lie group $G$, whose Lie algebra Lie$\left(  G\right)  $ is
isomorphic to $\mathfrak{g}.$

3. The isomorphism of $\mathfrak{g}$ with Lie$\left(  G\right)  ,$ explicitly,
means that if we define $Y_{\left[  I\right]  }$ as the left invariant (with
respect to $G$) vector field in $\mathbb{R}^{p}$ which agrees with
$\partial_{u_{I}}$ at the origin, then the Lie algebra generated by $\left\{
Y_{\left[  I\right]  }\right\}  _{I\in B}$ is isomorphic to $\mathfrak{g}$; in
particular, it is free up to weight $r$ and the vectors $\left\{  \left(
Y_{\left[  I\right]  }\right)  _{u}\right\}  _{\left\vert I\right\vert \leq
r}$ span $\mathbb{R}^{p}$ at any point. Clearly, this isomorphism logically
depends on the definition of $\circ$ in terms of the Baker-Campbell-Hausdorff formula.

4. It remains to show (iv) and (v). Both follow from taking a look inside the
operation $S\left(  x,y\right)  $. As to (v), from (\ref{S(x,y)}) we read that%
\[
S\left(  x,-x\right)  =x-x+\frac{1}{2}\left[  x,-x\right]  +\frac{1}%
{12}\left[  \left[  x,-x\right]  ,-x\right]  -\frac{1}{12}\left[  \left[
x,-x\right]  ,x\right]  +...=0
\]
since $\left[  x,x\right]  =0$. Therefore the Euclidean opposite is also the
inverse in the group. To prove (iv), we start writing, for any smooth function
$f$,%
\[
\left(  Y_{\left[  I\right]  }f\right)  \left(  x\right)  =\frac{d}%
{dt}_{\left/  _{t=0}\right.  }f\left(  x\circ te_{I}\right)
\]
by (\ref{S(x,y) Z}),
\begin{align*}
&  =\frac{d}{dt}_{\left/  _{t=0}\right.  }f\left(  \sum_{j+k\geq1}%
Z_{j,k}\left(  x,te_{I}\right)  \right)  =\nabla f\left(  x\right)  \cdot
\sum_{j+k\geq1}\frac{d}{dt}_{\left/  _{t=0}\right.  }t^{k}Z_{j,k}\left(
x,e_{I}\right) \\
&  =\nabla f\left(  x\right)  \cdot\sum_{j\geq0}Z_{j,1}\left(  x,e_{I}\right)
.
\end{align*}
Then we compute%
\begin{align*}
&  \sum_{I\in B}x_{I}\left(  Y_{\left[  I\right]  }f\right)  \left(  x\right)
=\sum_{I\in B}\nabla f\left(  x\right)  \cdot\sum_{j\geq0}Z_{j,1}\left(
x,x_{I}e_{I}\right)  =\\
&  =\nabla f\left(  x\right)  \cdot\sum_{j\geq0}Z_{j,1}\left(  x,x\right)
=\nabla f\left(  x\right)  \cdot x=\sum_{I\in B}x_{I}\partial_{x_{I}}f\left(
x\right)
\end{align*}
that is (\ref{Euler Y}). The theorem is completely proved. \medskip\hfill
\end{proof}

Theorem \ref{Thm Approximation local} can now applied choosing the left
invariant vector fields $Y_{\left[  I\right]  }$ as the approximating system.
The map $u=\Theta\left(  \eta,\xi\right)  $ can now be regarded as a
diffeomorphism from a neighborhood of $\eta$ onto a neighborhood of $0$
\textit{in the group} $G$. In other words, $\Theta\left(  \eta,\xi\right)  $
is an element of the group $G,$ and one has:%
\[
\Theta\left(  \eta,\xi\right)  =-\Theta\left(  \xi,\eta\right)  =\Theta\left(
\xi,\eta\right)  ^{-1}.
\]

\begin{remark}
Before stating their Theorem 5, H\"ormander and Melin suggest how to connect
it with Theorem 5 in \cite{RS}. We would like now to explain in more details
this link, giving in this way a reformulation of some of the results of
Theorem \ref{ilteoremabrutto}.

Let us consider the Lie algebra $\mathfrak{g}$ obtained by quotienting the
free Lie algebra with generators $X_{0},...,X_{n}$ with respect to the ideal
spanned by all commutators of weight greater than $r$. If $G$ is the (unique,
nilpotent) connected and simply connected Lie group having $\mathfrak{g}$ as
its Lie algebra, then the exponential map $\exp:\mathfrak{g}\rightarrow G$ is
a diffeomorphism, so that $\mathbb{R}^{p}\simeq\mathfrak{g}\simeq G$. We
denote with $Y_{i}\in\mathfrak{g}$ the equivalence class of $X_{i}$. It can be
seen, as usual, as a left-invariant vector field on $G$. It is also clear that
$\left\{  \left(  Y_{\left[  I\right]  }\right)  _{0}\right\}  _{I\in B}$ is a
basis of the tangent space $T_{0}G$, where $0$ here stands for the identity of
$G$. We can thus define a system of coordinates $\left(  u_{I}\right)  _{I\in
B}$ on $G$ by means of
\begin{equation}
\mathbb{R}^{p}\ni\left(  u_{I}\right)  _{I\in B}\mapsto\exp\left(  \sum_{I\in
B}u_{I}Y_{\left[  I\right]  }\right)  \left(  0\right)  =\exp\left(
\sum_{I\in B}u_{I}Y_{\left[  I\right]  }\right)  \in G, \label{ucooG}%
\end{equation}
as we did previously for the vector fields $X_{i}$ on $\mathbb{R}^{p}$. (The
first $\exp$ in (\ref{ucooG}) is the exponential of a vector field, while the
second one is the exponential map of the group $G$). Notice however that in
this case the coordinate system is global, since the exponential map is a
diffeomorphism. We used the same notation for the two sets of coordinates
because they will be immediately identified. Indeed, we can associate with
$\left(  u_{I}\right)  _{I\in B}\in\mathbb{R}^{p}$ a point in $\mathbb{R}^{p}$
by (\ref{can-coos}) and an element in $G$ by (\ref{ucooG}). This provides a
map from $G$ to $\mathbb{R}^{p}$, and allows us to compare the vector fields
$X_{\left[  I\right]  }$ and $Y_{\left[  I\right]  }$, once they have been
written in these coordinates. Put in a little bit different way, we can use
(\ref{ucooG}) to identify $G$ with the \textquotedblleft
same\textquotedblright\ $\mathbb{R}^{p}$ where the vector fields $X_{i}$ are
defined, so that the $X_{i}$ and the $Y_{i}$ live in the same space. At this
point, in order to apply Theorem \ref{Thm Approximation} we simply have to
notice that, arguing as in the proof of Lemma \ref{Euler}, one has
\[
\sum_{I\in B}u_{I}e_{I}=\sum_{I\in B}u_{I}Y_{\left[  I\right]  },
\]
where $e_{I}=\partial/\partial u_{I}$.
\end{remark}

\subsection{The ball-box theorem for free smooth vector
fields\label{sec ballbox}}

We now to draw some consequences from the study of weights of vector fields
(Theorem \ref{Thm weight of X_[I]}) in terms of the geometry of balls induced
by vector fields. We will get, still in the context of free smooth vector
fields, a ball-box theorem which is enough to get a control of the volume of
the balls in this setting. In turn, this fact will be exploited in the next
subsection to compare the distance induced by lifted vector fields with
Rothschild-Stein's quasidistance.

The subelliptic metric introduced by Nagel-Stein-Wainger in \cite{NSW}, in
this situation, is defined as follows:

\begin{definition}
\label{Definition CC distance}For any $\delta>0,$ let $C\left(  \delta\right)
$ be the class of absolutely continuous mappings $\varphi:\left[  0,1\right]
\longrightarrow\Omega$ which satisfy%
\[
\varphi^{\prime}\left(  t\right)  =\sum_{\left\vert I\right\vert \leq s}%
a_{I}\left(  t\right)  \left(  X_{\left[  I\right]  }\right)  _{\varphi\left(
t\right)  }\text{ a.e.}%
\]
with
\[
\left\vert a_{I}\left(  t\right)  \right\vert \leq\delta^{\left\vert
I\right\vert }.
\]
Then define%
\[
d\left(  x,y\right)  =\inf\left\{  \delta>0:\exists\varphi\in C\left(
\delta\right)  \text{ with }\varphi\left(  0\right)  =x,\varphi\left(
1\right)  =y\right\}  .
\]

\end{definition}

\begin{remark}
The quantity $d\left(  x,y\right)  $ is finite for any two points
$x,y\in\Omega.$ Namely, let $\varphi:\left[  0,1\right]  \longrightarrow
\Omega$ be any $C^{1}$ curve joining $x$ to $y$; since the $\left\{  \left(
X_{\left[  I\right]  }\right)  _{x}\right\}  _{\left\vert I\right\vert \leq
r}$ span $\mathbb{R}^{p}$ at any point $x\in\Omega,$ $\varphi^{\prime}\left(
t\right)  $ can always be expressed in the form
\[
\sum_{\left\vert I\right\vert \leq r}a_{I}\left(  t\right)  \left(  X_{\left[
I\right]  }\right)  _{\varphi\left(  t\right)  },
\]
for suitable bounded functions $a_{I}\left(  t\right)  ;$ then the curve
$\varphi$ will belong to the class $C\left(  \delta\right)  ,$ for $\delta>0$
large enough, and $d\left(  x,y\right)  $ will be finite and not exceeding
this $\delta$.
\end{remark}

\begin{proposition}
\label{Prop fefferman phong}The function $d:\Omega\times\Omega\rightarrow
\mathbb{R}$ is a distance. Moreover, for any $\Omega^{\prime}\Subset\Omega$
there exist positive constants $c_{1},c_{2}$ such that%
\begin{equation}
c_{1}\left\vert x-y\right\vert \leq d\left(  x,y\right)  \leq c_{2}\left\vert
x-y\right\vert ^{1/r}\text{ for any }x,y\in\Omega^{\prime}.
\label{fefferman-phong}%
\end{equation}

\end{proposition}

The previous proposition is well known (see \cite[Proposition 1.1]{NSW}; in
\cite{BBP} this is proved also for nonsmooth vector fields).

We are now in position to state our ball-box theorem.

\begin{notation}
For fixed $\overline{x}\in\mathbb{R}^{p},R>0,$ let%
\[
Box\left(  \overline{x},R\right)  =\left\{  x\in\mathbb{R}^{p}:x=\exp\left(
\sum_{I\in B}u_{I}X_{\left[  I\right]  }\right)  \left(  \overline{x}\right)
:\left\vert u_{I}\right\vert <R^{\left\vert I\right\vert }\text{ for any }I\in
B\right\}  ;
\]
In canonical coordinates $u_{I}$, the subset $Box\left(  \overline
{x},R\right)  $ simply becomes:%
\[
Box\left(  R\right)  =\left\{  u\in\mathbb{R}^{p}:\left\vert u_{I}\right\vert
<R^{\left\vert I\right\vert }\text{ for any }I\in B\right\}  .
\]
Let $B\left(  x,R\right)  $ denote the metric ball of center $x$ and radius
$R$ in $\mathbb{R}^{p}$, with respect to the distance $d$ induced by the
vector fields $\left\{  X_{\left[  I\right]  }\right\}  _{I\in B}$.

Also, let us define the following quantities related to canonical coordinates:%
\[
\left\vert u\right\vert _{k}=\sum_{\left\vert J\right\vert =k}\left\vert
u_{J}\right\vert \text{ for }k=1,2,...,s
\]%
\[
\left\Vert u\right\Vert =\sum_{k=1}^{s}\left\vert u\right\vert _{k}^{1/k}.
\]

\end{notation}

\begin{theorem}
[Ball-box theorem for free vector fields]\label{Thm ball-box free}For every
$\Omega^{\prime}\Subset\Omega$ there exist positive constants $C,R_{0}%
,c_{1},c_{2},c_{3}$ depending on $\Omega,\Omega^{\prime}$ and the system
$\left\{  X_{\left[  I\right]  }\right\}  _{I\in B}$ such that, for any
$\overline{x}\in\Omega^{\prime},R\leq R_{0},$

\begin{itemize}
\item[(i)]
\[
Box\left(  \overline{x},R\right)  \subseteq B\left(  \overline{x},R\right)
\subseteq Box\left(  \overline{x},CR\right)
\]

\item[(ii)]
\[
c_{1}R^{Q}\leq\left\vert B\left(  \overline{x},R\right)  \right\vert \leq
c_{2}R^{Q}%
\]
where $Q=\sum_{I\in B}\left\vert I\right\vert $ plays the role of
\textquotedblleft homogeneous dimension\textquotedblright

\item[(iii)]
\[
\left\vert B\left(  \overline{x},2R\right)  \right\vert \leq c_{3}\left\vert
B\left(  \overline{x},R\right)  \right\vert .
\]

\end{itemize}
\end{theorem}

\begin{proof}
(i) Let us show first that%
\begin{equation}
Box\left(  \overline{x},R\right)  \subseteq B\left(  \overline{x},R\right)  .
\label{Box c Ball}%
\end{equation}
For $x\in Box\left(  \overline{x},R\right)  ,$ let us write%
\begin{equation}
x=\exp\left(  \sum_{I\in B}u_{I}X_{\left[  I\right]  }\right)  \left(
\overline{x}\right)  \text{ with }\left\vert u_{I}\right\vert <R^{\left\vert
I\right\vert } \label{solito exp}%
\end{equation}
and set%
\[
\varphi\left(  t\right)  =\exp\left(  \sum_{I\in B}tu_{I}X_{\left[  I\right]
}\right)  \left(  \overline{x}\right)  .
\]
This $\varphi\left(  t\right)  $ defines an admissible curve belonging to
$C\left(  \delta\right)  $ for some $\delta<R$, that is, $d\left(
x,\overline{x}\right)  <R$ and inclusion (\ref{Box c Ball}) is proved.

To prove the reverse inclusion%
\[
B\left(  \overline{x},R\right)  \subseteq Box\left(  \overline{x},CR\right)
,
\]
we argue as follows. For $x\in B\left(  \overline{x},R\right)  ,$ let
$\varphi\left(  t\right)  $ be a curve in $C\left(  R\right)  $, that is,
\begin{align*}
\varphi^{\prime}\left(  t\right)   &  =\sum_{I\in B}a_{I}\left(  t\right)
\left(  X_{\left[  I\right]  }\right)  _{\varphi\left(  t\right)  }\\
\text{with }\varphi\left(  0\right)   &  =\overline{x},\varphi\left(
1\right)  =x,\left\vert a_{I}\left(  t\right)  \right\vert \leq R^{\left\vert
I\right\vert }.
\end{align*}
Then, for any smooth function $f\left(  x\right)  $ we have%
\[
f\left(  \varphi\left(  t\right)  \right)  -f\left(  \overline{x}\right)
=\int_{0}^{t}\frac{d}{dt}\left[  f\left(  \varphi\left(  \tau\right)  \right)
\right]  d\tau=\sum_{I\in B}\int_{0}^{t}a_{I}\left(  \tau\right)  \left(
X_{\left[  I\right]  }f\right)  _{\varphi\left(  \tau\right)  }d\tau.
\]
In particular, reasoning from now on in canonical coordinates, for $f\left(
u\right)  =u_{J}$ we get%
\begin{equation}
\varphi\left(  t\right)  _{J}=\sum_{I\in B}\int_{0}^{t}a_{I}\left(
\tau\right)  \left(  X_{\left[  I\right]  }u_{J}\right)  _{\varphi\left(
\tau\right)  }d\tau. \label{x_J}%
\end{equation}
From Theorem \ref{Thm weight of X_[I]} we read
\[
\left\vert \left(  X_{\left[  I\right]  }u_{J}\right)  _{\varphi\left(
\tau\right)  }\right\vert \leq c\left\Vert \varphi\left(  \tau\right)
\right\Vert ^{\left\vert J\right\vert -\left\vert I\right\vert }\text{ if
}\left\vert I\right\vert <\left\vert J\right\vert \ ,
\]
provided $x$ ranges in a compact set. Also, by definition of $\varphi$ we have%
\[
\left\vert a_{I}\left(  \tau\right)  \right\vert \leq R^{\left\vert
I\right\vert }\leq CR^{\left\vert J\right\vert }\text{ if }\left\vert
I\right\vert \geq\left\vert J\right\vert ,
\]
for any $R\leq R_{0},$ any fixed $R_{0},$ and some $C$ depending on $R_{0}$.
Therefore, (\ref{x_J}) gives%
\begin{align}
\left\vert \varphi\left(  t\right)  _{J}\right\vert _{k}  &  \leq C\sum_{I\in
B,\left\vert I\right\vert \leq k-1}\int_{0}^{t}R^{\left\vert I\right\vert
}\left\Vert \varphi\left(  \tau\right)  \right\Vert ^{\left\vert J\right\vert
-\left\vert I\right\vert }d\tau+\sum_{I\in B,\left\vert I\right\vert \geq
k}\int_{0}^{t}CR^{\left\vert J\right\vert }d\tau\nonumber\\
&  =C\left\{  \sum_{j=1}^{k-1}R^{j}\int_{0}^{t}\left\Vert \varphi\left(
\tau\right)  \right\Vert ^{k-j}d\tau+R^{k}\right\} \nonumber\\
&  \leq C\left\{  R\int_{0}^{t}\left\Vert \varphi\left(  \tau\right)
\right\Vert ^{k-1}d\tau+R^{k}\right\}  \label{(4)}%
\end{align}
where the last inequality holds because for $j=1,2,...,k-1$%
\[
R^{j}\left\Vert x\right\Vert ^{k-j}\leq\left\{
\begin{array}
[c]{l}%
R^{k}\text{ if }\left\Vert x\right\Vert \leq R\\
R\left\Vert x\right\Vert ^{k-1}\text{ if }\left\Vert x\right\Vert \geq R.
\end{array}
\right.
\]
Next, since%
\[
R\left\Vert \varphi\left(  \tau\right)  \right\Vert ^{k-1}\leq\left(
R+\left\Vert \varphi\left(  \tau\right)  \right\Vert \right)  ^{k}\leq
c\left(  R^{k}+\left\Vert \varphi\left(  \tau\right)  \right\Vert ^{k}\right)
,
\]
from (\ref{(4)}) we get%
\[
\left\vert \varphi\left(  t\right)  _{J}\right\vert _{k}\leq C\left\{
\int_{0}^{t}\left\Vert \varphi\left(  \tau\right)  \right\Vert ^{k}d\tau
+R^{k}\right\}  .
\]
Recalling the definition of $\left\Vert \cdot\right\Vert $ we then have%
\begin{align*}
\left\vert \varphi\left(  t\right)  _{J}\right\vert _{k}^{1/k}  &  \leq
C\left\{  \left(  \int_{0}^{t}\left\Vert \varphi\left(  \tau\right)
\right\Vert ^{k}d\tau\right)  ^{1/k}+R\right\} \\
\left\Vert \varphi\left(  \tau\right)  \right\Vert  &  \leq C\sum_{k=1}%
^{r}\left\{  \left(  \int_{0}^{t}\left\Vert \varphi\left(  \tau\right)
\right\Vert ^{k}d\tau\right)  ^{1/k}+R\right\} \\
\left\Vert \varphi\left(  \tau\right)  \right\Vert ^{r}  &  \leq C\sum
_{k=1}^{r}\left\{  \left(  \int_{0}^{t}\left\Vert \varphi\left(  \tau\right)
\right\Vert ^{k}d\tau\right)  ^{r/k}+R^{s}\right\}
\end{align*}
by H\"{o}lder inequality, since $r/k>1,$%
\[
\leq C\sum_{k=1}^{r}\left\{  \int_{0}^{t}\left\Vert \varphi\left(
\tau\right)  \right\Vert ^{r}d\tau+R^{r}\right\}  .
\]
Hence Gronwall's inequality implies%
\[
\left\Vert \varphi\left(  \tau\right)  \right\Vert ^{r}\leq CR^{r}\text{ for
any }\tau<t\leq1,
\]
which for $\tau=1$ gives%
\[
\left\Vert x\right\Vert \leq CR.
\]
that is $x\in Box\left(  CR\right)  $.

(ii). Let
\[
F\left(  u,x\right)  =\exp\left(  \sum_{I\in B}u_{I}X_{\left[  I\right]
}\right)  \left(  x\right)  ,\text{ for }\left(  u,x\right)  \in
U\times\overline{\Omega^{\prime}}%
\]
for some neighborhood $U$ of $0$ and let $J\left(  u,x\right)  $ be the
Jacobian determinant of the map $u\mapsto F\left(  u,x\right)  $. Since%
\[
J\left(  0,x\right)  =\det\left(  \left(  X_{\left[  I\right]  }\right)
_{x}\right)  _{I\in B}%
\]
by compactness $J\left(  u,x\right)  $ is bounded and bounded away from zero
in $U^{\prime}\times\overline{\Omega^{\prime}},$ for a suitable open subset
$U^{\prime}\subset U.$ Therefore,%
\begin{align*}
\left\vert B\left(  x,R\right)  \right\vert  &  \leq\left\vert Box\left(
\overline{x},CR\right)  \right\vert =\int_{Box\left(  \overline{x},CR\right)
}dy=\\
&  =\int_{\left\vert u_{I}\right\vert <\left(  CR\right)  ^{\left\vert
I\right\vert }}\left\vert J\left(  u,x\right)  \right\vert du\leq
c\int_{\left\vert u_{I}\right\vert <\left(  CR\right)  ^{\left\vert
I\right\vert }}du=cR^{Q},
\end{align*}
and analogously we establish the reverse inequality.

Finally, (iii) immediately follows from (ii).
\end{proof}

\begin{remark}
The above proof basically relies on Theorem \ref{Thm weight of X_[I]} and
elementary facts. However, we want also to stress the fact that the uniform
control on the constants is possible since the vector fields are free, so that
a basis can be chosen once and for all, independently from the point.
\end{remark}

\subsection{Equivalent quasidistances for free vector fields}

Let us consider again the free lifted vector fields ${X}_{0},{X}_{1}%
,...,{X}_{n},$ and the map $\Theta\left(  \xi,\eta\right)  ,$ defined for
$\xi,\eta$ belonging to a suitable neighborhood $W$ of a fixed point $\xi
_{0}.$ Recall that $u=\Theta\left(  \xi,\eta\right)  $ can be seen as an
element of the group $G.$ A point $u\in G$ is individuated by coordinates
$\left\{  u_{J}\right\}  _{J\in B}$. One can define%
\[
\rho\left(  \xi,\eta\right)  =\left\Vert \Theta\left(  \xi,\eta\right)
\right\Vert ,
\]
the Rothschild-Stein's \textit{quasidistance }induced by the ${X}_{i}$. It is
defined only locally and satisfies the properties collected in the next:

\begin{proposition}
For every $\xi_{0}\in\mathbb{R}^{p}$ there exist a neighborhood $W$ of
$\xi_{0}$ and constants $c,c_{1},c_{2}>0$ such that for any $\xi,\eta\in W,$%
\begin{align}
\rho\left(  \xi,\eta\right)   &  \geq0\nonumber\\
\rho\left(  \xi,\eta\right)   &  =0\Longleftrightarrow\xi=\eta\nonumber\\
\rho\left(  \xi,\eta\right)   &  =\rho\left(  \eta,\xi\right) \nonumber\\
\rho\left(  \xi,\eta\right)   &  \leq c\left\{  \rho\left(  \xi,\zeta\right)
+\rho\left(  \zeta,\eta\right)  \right\} \label{quasi triangle inequality}\\
c_{1}d\left(  \xi,\eta\right)   &  \leq\rho\left(  \xi,\eta\right)  \leq
c_{2}d\left(  \xi,\eta\right)  . \label{equiv d rho}%
\end{align}

\end{proposition}

\begin{proof}
The first three properties follow by definition and by Proposition
\ref{Map Theta}, while (\ref{quasi triangle inequality}) follows by
(\ref{equiv d rho}), since $d$ satisfies the triangle inequality. So let us
prove (\ref{equiv d rho}).

Let us denote by $B_{\rho}\left(  \xi,R\right)  $ the \textquotedblleft
balls\textquotedblright\ with respect to $\rho$. Note that, just by definition
of box and $\Theta,$ we have
\[
\eta\in Box\left(  \xi,R\right)  \Longleftrightarrow\left\Vert \Theta\left(
\xi,\eta\right)  \right\Vert \leq R,
\]
which implies the inclusions%
\[
B_{\rho}\left(  \xi,c_{1}R\right)  \subset Box\left(  \xi,R\right)  \subset
B_{\rho}\left(  \xi,c_{2}R\right)
\]
for any $\xi\in W$, $R\leq R_{0}$, and some positive constants $R_{0}%
,c_{1},c_{2}.$ Therefore, Theorem \ref{Thm ball-box free} implies
(\ref{equiv d rho}).
\end{proof}

We also have the following

\begin{proposition}
The change of coordinate in $\mathbb{R}^{p}$ given by%
\[
\xi\mapsto u=\Theta\left(  \xi,\eta\right)
\]
has a Jacobian determinant given by%
\[
d\xi=c\left(  \eta\right)  \left(  1+O\left(  \left\Vert u\right\Vert \right)
\right)  du
\]
where $c\left(  \eta\right)  $ is a smooth function, bounded and bounded away
from zero.
\end{proposition}

The above proposition is proved in \cite{RS}; see also \cite[Thm. 1.7]{BB2}.
Moreover, this proposition is a particular case of the analogous property
which we will prove for nonsmooth vector fields in subsection
\ref{section nonsmooth theta}, Proposition \ref{Prop nonsmooth change}.

\section{Approximation for nonsmooth H\"{o}rmander's vector
fields\label{Section nonsmooth}}

Here we want to prove also for nonsmooth vector fields the approximation
theorem and the basic results about the map $\Theta_{\eta}\left(
\cdot\right)  $. This is made possible combining the previous theory for
smooth vector fields with a natural procedure of approximation of nonsmooth
vector fields by their Taylor expansion. To quantify the weight of the
\textquotedblleft remainders\textquotedblright\ in the approximation formula,
we have to assume the coefficients of the vector fields in the scale of
H\"{o}lder spaces, slightly strengthening the assumptions made in
\S \ref{subsection lifting}.

\subsection{Assumptions}

\textbf{Assumptions (B). }We assume that for some integer $r\geq2,$ some
$\alpha\in(0,1]$ and some bounded domain $\Omega\subset\mathbb{R}^{p}$ the
following hold:

\begin{itemize}
\item[(B1)] The coefficients of the vector fields $X_{1},X_{2},...,X_{n}$
belong to $C^{r-1\,,\alpha}\left(  \Omega\right)  ,$ while the coefficients of
$X_{0}$ belong to $C^{r-2,\alpha\,}\left(  \Omega\right)  .$ Here and in the
following, $C^{k,\alpha}$ stands for the classical space of functions with
derivatives up to order $k$, H\"{o}lder continuous of exponent $\alpha.$

\item[(B2)] The vectors $\left\{  \left(  X_{\left[  I\right]  }\right)
_{x}\right\}  _{\left\vert I\right\vert \leq r}$ span $\mathbb{R}^{p}$ at
every point $x\in\Omega$.
\end{itemize}

The following easy lemma is proved analogously to that in
\S \ref{subsec assumptions}.

\begin{lemma}
Under the assumption (B1) above, for any $1\leq k\leq r,$ the differential
operators%
\[
\left\{  X_{I}\right\}  _{\left\vert I\right\vert \leq k}%
\]
are well defined, and have $C^{r-k,\alpha}$ coefficients. The same is true for
the vector fields $\left\{  X_{\left[  I\right]  }\right\}  _{\left\vert
I\right\vert \leq k}.$
\end{lemma}

\textbf{Dependence of the constants.} We will often write that some constant
depends on the vector fields $X_{i}$'s and some fixed domain $\Omega^{\prime
}\Subset\Omega$. (Actually, the dependence on the $X_{i}$'s will be usually
left understood). Explicitly, this will mean that the constant depends on:

(i) $\Omega^{\prime}$;

(ii) the norms $C^{r-1,\alpha}\left(  \Omega\right)  $ of the coefficients of
$X_{i}$ $\left(  i=1,2,...,n\right)  $ and the norms $C^{r-2,\alpha}\left(
\Omega\right)  $ of the coefficients of $X_{0}$;

(iii) a positive constant $c_{0}$ such that the following bound holds:%
\[
\inf_{x\in\Omega^{\prime}}\max_{\left\vert I_{1}\right\vert ,\left\vert
I_{2}\right\vert ,...,\left\vert I_{p}\right\vert \leq r}\left\vert
\det\left(  \left(  X_{\left[  I_{1}\right]  }\right)  _{x},\left(  X_{\left[
I_{2}\right]  }\right)  _{x},...,\left(  X_{\left[  I_{p}\right]  }\right)
_{x}\right)  \right\vert \geq c_{0}%
\]
(where \textquotedblleft$\det$\textquotedblright\ denotes the determinant of
the $p\times p$ matrix having the vectors $\left(  X_{\left[  I_{i}\right]
}\right)  _{x}$ as rows).

Note that (iii) is a quantitative way of assuring the validity of
H\"{o}rmander's condition, uniformly in $\Omega^{\prime}$.

As we have seen in \S \ref{subsection lifting}, we can always lift our
nonsmooth vector fields to get a system of free vector fields, still
satisfying assumptions (B). In this section we are interested in proving a
Rothschild-Stein type approximation result for these lifted vector fields.
Therefore, just to simplify notation, throughout this section we will assume
that our vector fields $X_{i}$'s are already free up to weight $r$. For the
same reason, we will use variables $x,y$ instead of $\xi,\eta$.

\subsection{Regularized canonical
coordinates\label{subsec Regularized canonical coordinates}}

As we have seen in \S \ref{subsection canonical coordinates}, of basic
importance in the study of the properties of the vector fields is expressing
them in terms of \textit{canonical coordinates}. This means to introduce the
local diffeomorphism
\[
\mathbb{R}^{p}\backepsilon u\longmapsto x=\exp\left(  \sum_{I\in B}%
u_{I}X_{\left[  I\right]  }\right)  \left(  \overline{x}\right)  \in U\left(
\overline{x}\right)
\]
and then express the vector fields $X_{\left[  I\right]  }^{x}$ as
differential operators $X_{\left[  I\right]  }^{u}$ acting on the variable
$u$. However, under our assumptions, we cannot expect this map being more
regular than H\"{o}lder continuous; even if we strengthened our assumptions
asking the coefficients of $X_{i}$ to be $C^{r-p_{i}+1,\alpha}$, we would get
a $C^{1,\alpha}$ local diffeomorphism, which would transform the vector fields
$X_{i}^{x}$ in $C^{\alpha}$ vector fields $X_{i}^{u}$. Therefore it would be
impossible to compute the commutators of the transformed vector fields, and
all the arguments of \S \ref{subsection canonical coordinates} would break
down. More precisely, following this line we would be forced to require the
coefficients of $X_{i}$ to be $C^{2r,\alpha},$ which is rather unsatisfactory.
This discussion leads us to look for a smooth diffeomorphism, adapted to the
system $\left\{  X_{\left[  I\right]  }\left(  x\right)  \right\}
_{\left\vert I\right\vert \leq r},$ transforming the vector fields $X_{\left[
I\right]  }$ in vector fields $X_{\left[  I\right]  }^{u}$ having the same
regularity, and better properties. This leads to the concept of
\textit{regularized canonical coordinates}, firstly introduced in \cite{C} and
then used by several authors in particular cases.

Fix a point $\overline{x}\in\Omega;$ for any $i=0,1,2,...,n$, let us consider
the vector field
\[
X_{i}=\sum_{j=1}^{p}b_{ij}\left(  x\right)  \partial_{x_{j}};
\]
let $p_{ij}^{r}\left(  x\right)  $ be the Taylor polynomial of $b_{ij}\left(
x\right)  $ of center $\overline{x}$ and order $r-p_{i}$; note that%
\begin{equation}
b_{ij}\left(  x\right)  =p_{ij}^{r}\left(  x\right)  +O\left(  \left\vert
x-\overline{x}\right\vert ^{r-p_{i}+\alpha}\right)  ;
\label{Holder approximation}%
\end{equation}
set%
\[
S_{i}^{\overline{x}}=\sum_{j=1}^{p}p_{ij}^{r}\left(  x\right)  \partial
_{x_{j}}.
\]
From (\ref{Holder approximation}) we easily have (see, e.g., \cite{BBP}):

\begin{proposition}
\label{Proposition S_i}The $S_{i}^{\overline{x}}$ $\left(
i=0,1,2,...,n\right)  $ are smooth vector fields defined in the whole space,
satisfying:%
\[
\left(  S_{I}^{\overline{x}}\right)  _{\overline{x}}=\left(  X_{I}\right)
_{\overline{x}}\text{ and }\left(  S_{\left[  I\right]  }^{\overline{x}%
}\right)  _{\overline{x}}=\left(  X_{\left[  I\right]  }\right)
_{\overline{x}}\text{ for any }I\text{ with }\left\vert I\right\vert \leq
r\text{.}%
\]
Moreover,%
\[
X_{\left[  I\right]  }-S_{\left[  I\right]  }^{\overline{x}}=\sum_{j=1}%
^{p}c_{I}^{j}\left(  x\right)  \partial_{x_{j}}\text{ with }c_{I}^{j}\left(
x\right)  =O\left(  \left\vert x-\overline{x}\right\vert ^{r-\left\vert
I\right\vert +\alpha}\right)  .
\]
Finally, denoting by $d_{X}$ and $d_{S^{\overline{x}}}$ the distances (see
\S \ \ref{sec ballbox}) induced by the $X_{i}$'s and the $S_{i}^{\overline{x}%
}$'s, respectively, and by $B_{X}$ and $B_{S^{\overline{x}}}$ the
corresponding metric balls, there exist positive constants $c_{1},c_{2},R_{0}$
depending on $\Omega,\Omega^{\prime}$ and the $X_{i}$'s$,$ but not on
$\overline{x}\in\Omega^{\prime}$, such that%
\[
B_{S^{\overline{x}}}\left(  \overline{x},c_{1}R\right)  \subset B_{X}\left(
\overline{x},R\right)  \subset B_{S^{\overline{x}}}\left(  \overline{x}%
,c_{2}R\right)
\]
for any $R<R_{0}.$
\end{proposition}

Now, fix a point $\overline{x}\in\Omega,$ and select a basis of $\mathbb{R}%
^{p}$ of the form%
\[
\left\{  \left(  X_{\left[  I\right]  }\right)  _{\overline{x}}\right\}
_{I\in B}.
\]
Clearly, we have
\[
\left\{  \left(  X_{\left[  I\right]  }\right)  _{\overline{x}}\right\}
_{I\in B}=\left\{  \left(  S_{\left[  I\right]  }\right)  _{\overline{x}%
}\right\}  _{I\in B}.
\]
We can now introduce, as in \S \ref{subsection canonical coordinates}, the
canonical coordinates induced by the smooth vector fields $S_{i}^{\overline
{x}}$; these will be, by definition, the regularized canonical coordinates
induced by the $X_{i}$'s:%
\begin{equation}
\mathbb{R}^{p}\backepsilon u\longmapsto x=\exp\left(  \sum_{I\in B}%
u_{I}S_{\left[  I\right]  }^{\overline{x}}\right)  \left(  \overline
{x}\right)  \label{diffeomorphism}%
\end{equation}
for $x$ belonging to some neighborhood $U\left(  \overline{x}\right)  .$ Note
that the Jacobian of the map $u\longmapsto x,$ at $u=0,$ equals the matrix of
the vector fields $\left\{  \left(  S^{\overline{x}}_{\left[  I\right]
}\right)  _{\overline{x}}\right\}  _{I\in B}=\left\{  \left(  X_{\left[
I\right]  }\right)  _{\overline{x}}\right\}  _{I\in B},$ therefore is
nonsingular. Moreover, since the\ $S^{\overline{x}}_{\left[  I\right]  }$'s
are smooth, the diffeomorphism is smooth, too.

Next, we express our original vector fields $X_{i}$ in terms of regularized
canonical coordinates $u$: let us write $X_{\left[  I\right]  }^{u}$ to denote
the vector field $X_{\left[  I\right]  }$ expressed in coordinates $u$. The
following facts are immediate:

\begin{proposition}

\begin{itemize}
\item[(i)] The (transformed) vector field $X_{i}^{u}$ has $C^{r-p_{i},\alpha}$
coefficients, for $i=0,1,2,...,n$;

\item[(ii)] the vector field $X_{\left[  I\right]  }^{u}$ has $C^{r-\left\vert
I\right\vert ,\alpha}$ coefficients, for any $I\ $such that $\left\vert
I\right\vert \leq r;$ in particular, all the $\left\{  X_{\left[  I\right]
}^{u}\right\}  _{I\in B}$ have $C^{\alpha}$ coefficients;

\item[(iii)]
\[
\left[  X_{\left[  I\right]  },X_{\left[  J\right]  }\right]  ^{u}=\left[
X_{\left[  I\right]  }^{u},X_{\left[  J\right]  }^{u}\right]
\]
for any $I,J$ such that $\left\vert I\right\vert +\left\vert J\right\vert \leq
r$.
\end{itemize}
\end{proposition}

\subsection{Weights and approximation}

Using the coordinates $u$ we can give the following

\begin{definition}
Let $X$ be a vector field with possibly nonsmooth coefficients. We will say
that $X$ has weight $\geq k\in\mathbb{R}$ near the point $\overline{x}$ if,
expressing it in regularized canonical coordinates
\[
X^{u}=\sum_{J\in B}c_{J}\left(  u\right)  \partial_{u_{J}}%
\]
we have:%
\[
\left\vert c_{J}\left(  u\right)  \right\vert \leq c\left\Vert u\right\Vert
^{k+\left\vert J\right\vert }\text{.}%
\]
for $u$ in a neighborhood of $0$.
\end{definition}

Here, as in \S \ref{subsection canonical coordinates},%
\[
\left\Vert u\right\Vert =\sum_{J\in B}\left\vert u_{J}\right\vert
^{1/\left\vert J\right\vert }.
\]
Note that if $X$ is a smooth vector field of weight $\geq k\in\mathbb{Z},$ in
the sense of Definition \ref{Definition weights}, then it is also of weight
$\geq k$ in the sense of the above definition. We will also use the following
elementary remark:

if $X,Y$ have weight $\geq k_{X},k_{Y},$ respectively, then $X\pm Y$ has
weight $\geq\min\left(  k_{X},k_{Y}\right)  .$

\begin{proposition}
\label{Proposition X_I-S_I}The vector field $X_{\left[  I\right]
}-S^{\overline{x}}_{\left[  I\right]  }$ has weight $\geq\alpha-\left\vert
I\right\vert $ near $\overline{x},$ for any $\left\vert I\right\vert \leq r.$
\end{proposition}

\begin{proof}
By Proposition \ref{Proposition S_i}, we know that%
\begin{equation}
X_{\left[  I\right]  }-S^{\overline{x}}_{\left[  I\right]  }=\sum_{j=1}%
^{p}c_{I}^{j}\left(  x\right)  \partial_{x_{j}}\text{ with }c_{I}^{j}\left(
x\right)  =O\left(  \left\vert x-\overline{x}\right\vert ^{r-\left\vert
I\right\vert +\alpha}\right)  . \label{X_I-S_I}%
\end{equation}
Let
\[
u=F\left(  x\right)  =\left(  F_{J}\left(  x\right)  \right)  _{J\in B}%
\]
be the local smooth diffeomorphism defined as in (\ref{diffeomorphism}), and
let $x=F^{-1}\left(  u\right)  $ be its inverse. Then vector fields are
transformed according to the law:%
\[
\partial_{x_{j}}=\sum_{J\in B}\left(  \partial_{x_{j}}F_{J}\right)  \left(
x\right)  \partial_{u_{J}}.
\]
Therefore%
\[
X_{\left[  I\right]  }^{u}-\left(  S^{\overline{x}}_{\left[  I\right]
}\right)  ^{u}=\sum_{J\in B}\sum_{j=1}^{p}\left(  c_{I}^{j}\partial_{x_{j}%
}F_{J}\right)  \left(  x\right)  \partial_{u_{J}}=\sum_{J\in B}\widetilde
{c}_{I}^{J}\left(  u\right)  \partial_{u_{J}}%
\]
with%
\[
\left\vert \widetilde{c}_{I}^{J}\left(  u\right)  \right\vert =\left\vert
\sum_{j=1}^{p}\left(  c_{I}^{j}\partial_{x_{j}}F_{J}\right)  \left(  x\right)
\right\vert \leq
\]
since $F$ is a smooth diffeomorphism and by (\ref{X_I-S_I})%
\[
\leq c\sum_{j=1}^{p}\left\vert c_{I}^{j}\left(  x\right)  \right\vert \leq
c\left\vert x-\overline{x}\right\vert ^{r-\left\vert I\right\vert +\alpha}\leq
c\left\Vert u\right\Vert ^{r-\left\vert I\right\vert +\alpha}\leq c\left\Vert
u\right\Vert ^{\left\vert J\right\vert -\left\vert I\right\vert +\alpha}.
\]
This ends the proof.
\end{proof}

\begin{proposition}
\label{Prop weight X-Y}For any $\left\vert I\right\vert \leq r$ we have:

\begin{enumerate}
\item[(i)] the vector field $X_{\left[  I\right]  }$ has weight $\geq
-\left\vert I\right\vert $ near $\overline{x}$.

\item[(ii)] If $Y_{0},Y_{1},\ldots,Y_{n}$ is any system of smooth vector
fields satisfying (with respect to regularized canonical coordinates)%
\[
\sum_{I\in B}u_{I}e_{I}=\sum_{I\in B}u_{I}Y_{\left[  I\right]  }%
\]
then $X_{\left[  I\right]  }-Y_{\left[  I\right]  }$ has weight $\geq
\alpha-\left\vert I\right\vert $.
\end{enumerate}
\end{proposition}

\begin{proof}
We can apply to the system of smooth vector fields $\left\{  S^{\overline{x}%
}_{\left[  I\right]  }\right\}  _{I\in B}$ the theory developed in Section
\ref{section approximation} and say that:

\begin{enumerate}
\item[(i)] the vector field $S^{\overline{x}}_{\left[  I\right]  }$ has weight
$\geq-\left\vert I\right\vert $.

\item[(ii)] $S^{\overline{x}}_{\left[  I\right]  }-Y_{\left[  I\right]  }$ has
weight $\geq1-\left\vert I\right\vert $.
\end{enumerate}

Assertion (ii) exploits the fact that the $S^{\overline{x}}_{i}$'s are free up
to weight $r$ at $\overline{x},$ if the $X_{i}$'s are so, because the $X_{i}%
$'s and the $S^{\overline{x}}_{i}$'s satisfy the same commutation relations,
up to weight $r$, at $\overline{x}.$

Therefore, by Proposition \ref{Proposition X_I-S_I}, we conclude that:

\begin{enumerate}
\item[(i)] $X_{\left[  I\right]  }=S^{\overline{x}}_{\left[  I\right]
}+\left(  X_{\left[  I\right]  }-S^{\overline{x}}_{\left[  I\right]  }\right)
$ has weight $\geq\min\left(  -\left\vert I\right\vert ,-\left\vert
I\right\vert +\alpha\right)  =-\left\vert I\right\vert $

\item[(ii)] $X_{\left[  I\right]  }-Y_{\left[  I\right]  }=\left(  X_{\left[
I\right]  }-S_{\left[  I\right]  }^{\overline{x}}\right)  +\left(  S_{\left[
I\right]  }^{\overline{x}}-Y_{\left[  I\right]  }\right)  $ has weight
$\geq\min\left(  \alpha-\left\vert I\right\vert ,1-\left\vert I\right\vert
\right)  =\alpha-\left\vert I\right\vert .$\medskip
\end{enumerate}
\end{proof}

We now need some quantitative information about the dependence of our
approximation procedure on the point $\overline{x}$.

\begin{proposition}
\label{NonSmoothExpMap}Let $E\left(  u,\overline{x}\right)  =\exp\left(
\sum_{I\in B}u_{I}S_{\left[  I\right]  }^{\overline{x}}\right)  \left(
\overline{x}\right)  $. Then for every multi-index $\beta$ the derivative
$\frac{\partial^{\left\vert \beta\right\vert }E}{\partial u^{\beta}}\left(
u,\overline{x}\right)  $ depends on $\overline{x}$ in a $C^{\alpha}$ way.
\end{proposition}

\begin{proof}
We know that%
\[
E\left(  u,\overline{x}\right)  =\gamma\left(  1,u,\overline{x}\right)
\]
where $\gamma$ solves the Cauchy problem%
\[
\left\{
\begin{array}
[c]{l}%
\frac{d}{dt}\gamma\left(  t,u,\overline{x}\right)  =\sum_{I\in B}u_{I}\left(
S_{\left[  I\right]  }^{\overline{x}}\right)  _{\gamma\left(  t,u,\overline
{x}\right)  }\\
\gamma\left(  0,u,\overline{x}\right)  =\overline{x}.
\end{array}
\right.
\]
For a fixed $\overline{x}$ the solution $\gamma\left(  \cdot,\cdot
,\overline{x}\right)  $ is smooth; moreover $\gamma$ depends on $\overline{x}$
in a $C^{\alpha}$ way. Therefore%
\[
\left\{
\begin{array}
[c]{l}%
\frac{d}{dt}\frac{\partial\gamma}{\partial u_{J}}\left(  t,u,\overline
{x}\right)  =\sum_{I\in B}u_{I}\sum_{j=1}^{p}\left(  \frac{\partial S_{\left[
I\right]  }^{\overline{x}}}{\partial x_{j}}\left(  \gamma\left(
t,u,\overline{x}\right)  \right)  \frac{\partial\gamma_{j}}{\partial u_{J}%
}\left(  t,u,\overline{x}\right)  \right)  +\left(  S_{\left[  J\right]
}^{\overline{x}}\right)  _{\gamma\left(  t,u,\overline{x}\right)  }\\
\frac{\partial\gamma}{\partial u_{J}}\left(  0,u,\overline{x}\right)  =0.
\end{array}
\right.
\]
Let now%
\begin{align*}
\omega\left(  t,u,\overline{x}\right)   &  =\frac{\partial\gamma}{\partial
u_{J}}\left(  t,u,\overline{x}\right)  ,\\
A\left(  t,u,\overline{x}\right)   &  =\sum_{I\in B}u_{I}~\frac{\partial
S_{\left[  I\right]  }^{\overline{x}}}{\partial x}\left(  \gamma\left(
t,u,\overline{x}\right)  \right) \\
B_{J}\left(  t,u,\overline{x}\right)   &  =\left(  S_{\left[  J\right]
}^{\overline{x}}\right)  _{\gamma\left(  t,u,\overline{x}\right)  }%
\end{align*}
Since $\left(  S_{\left[  J\right]  }^{\overline{x}}\right)  _{x}$ and
$\frac{\partial S_{\left[  I\right]  }^{\overline{x}}}{\partial x_{j}}\left(
x\right)  $ are smooth in the $x$ variable and $C^{\alpha}$ in the
$\overline{x}$ variable, the functions $A\left(  t,u,\overline{x}\right)  $
and $B_{J}\left(  t,u,\overline{x}\right)  $ are smooth in $\left(
t,u\right)  $ and $C^{\alpha}$ in $\overline{x}$. With the above notation,%
\begin{align*}
\frac{d}{dt}\omega\left(  t,u,\overline{x}\right)   &  =A\left(
t,u,\overline{x}\right)  ~\omega\left(  t,u,\overline{x}\right)  +B_{J}\left(
t,u,\overline{x}\right) \\
\omega\left(  0,u,\overline{x}\right)   &  =0
\end{align*}
from where we readily see that $\omega$ is $C^{\alpha}$ in $\overline{x}$.
This shows that $\frac{\partial E}{\partial u_{J}}\left(  u,\overline
{x}\right)  =\omega\left(  1,u,\overline{x}\right)  $ has the same property.
An iteration of this argument yields the desired results.
\end{proof}

To make more usable Proposition \ref{Prop weight X-Y} we have to construct, as
in the smooth case, a map $u=\Theta\left(  \overline{x},x\right)  ,$ allowing
to compute the derivative $X_{\left[  I\right]  }f\left(  x\right)  $ without
passing to variables $u$.

Let $E\left(  u,\overline{x}\right)  $ as in Proposition \ref{NonSmoothExpMap}%
. Clearly, for any fixed $\overline{x},$ the map $u\longmapsto E\left(
u,\overline{x}\right)  $ is smooth. Moreover, its Jacobian determinant at
$u=0$ equals%
\[
\det\left(  S_{\left[  I\right]  ,\overline{x}}\right)  \left(  \overline
{x}\right)  \neq0
\]
because $\left\{  \left(  S_{\left[  I\right]  ,\overline{x}}\right)
_{\overline{x}}\right\}  _{I\in B}$ is a basis of $\mathbb{R}^{p}.$ Therefore
there exists a smooth inverse function, which we denote by%
\[
u=\Theta_{\overline{x}}\left(  x\right)  .
\]
A basic difference with the smooth theory is that $\Theta_{\overline{x}%
}\left(  x\right)  $ is not simply $-\Theta_{x}\left(  \overline{x}\right)  $.
This is due to the fact that if $x=E\left(  u,\overline{x}\right)  $ then%
\[
E\left(  -u,x\right)  =\exp\left(  -\sum_{I\in B}u_{I}S_{\left[  I\right]
}^{x}\right)  \exp\left(  \sum_{I\in B}u_{I}S_{\left[  I\right]  }%
^{\overline{x}}\right)  \left(  \overline{x}\right)  \neq\overline{x}%
\]
because the vector fields in the first exponential are $S_{\left[  I\right]
}^{x}$ while those in the second one are $S_{\left[  I\right]  }^{\overline
{x}}$. Due to this asymmetry, we cannot expect $\Theta_{\overline{x}}\left(
x\right)  $ to be as smooth in the $\overline{x}$ variable as it is in the $x$
variable. Instead, since the vector fields depend on $\overline{x}$ in a
$C^{\alpha}$ continuous way, the best we can hope is $C^{\alpha}$ continuity
with respect to $\overline{x}$ also for $\Theta_{\overline{x}}\left(
x\right)  .$ This is actually the case:

\begin{proposition}
\label{Prop Theta C alfa}For any fixed $x,$ the map $\overline{x}\mapsto
\Theta_{\overline{x}}\left(  x\right)  $ is $C^{\alpha}$. Moreover, each entry
of the Jacobian matrix
\[
\frac{\partial\Theta_{\overline{x}}\left(  x\right)  }{\partial x}%
\]
still depends on $\overline{x}$ in a $C^{\alpha}$ way.
\end{proposition}

\begin{proof}
By the above Proposition the function $\overline{x}\mapsto E\left(
u,\overline{x}\right)  $ is $C^{\alpha}$ continuous. We are now going to
revise the proof of the inverse function theorem, showing that this implies a
$C^{\alpha}$ dependence on the same parameter $\overline{x}$ for the inverse
function $x\mapsto\Theta_{\overline{x}}\left(  x\right)  $.

Let $A_{\overline{x}}$ be the Jacobian matrix $\partial_{u}E\left(
u,\overline{x}\right)  $ evaluated at $u=0$, and, for a fixed $x,$ set%
\[
\varphi_{\overline{x}}\left(  u\right)  =u+A_{\overline{x}}^{-1}\left(
E\left(  u,\overline{x}\right)  -x\right)  .
\]
To find the inverse function of $u\mapsto E\left(  u,\overline{x}\right)  $ we
look for a fixed point of $\varphi_{\overline{x}}.$ Since $E\left(
u,\overline{x}\right)  $ is a smooth function of $u$ and $E\left(
0,\overline{x}\right)  -\overline{x}=0,$ for $u$ in a suitably small
neighborhood of $0$ and $x$ in a suitable neighborhood of $\overline{x},$
$\varphi_{\overline{x}}$ is a contraction; under these assumptions, we can
write
\[
\left\vert \varphi_{\overline{x}}\left(  u_{1}\right)  -\varphi_{\overline{x}%
}\left(  u_{1}\right)  \right\vert \leq\delta\left\vert u_{1}-u_{2}%
\right\vert
\]
for some $\delta\in\left(  0,1\right)  .$ For any two points $x_{1},x_{2}$ in
a small neighborhood of $\overline{x},$ let us define the sequence:%
\[
\left\{
\begin{array}
[c]{l}%
u_{n+1}^{x_{i}}=\varphi_{x_{i}}\left(  u_{n}\right) \\
u_{0}^{xi}=0
\end{array}
\right.
\]
for $i=1,2.$ Clearly, $u_{n}^{x_{i}}\rightarrow E\left(  \cdot,x_{i}\right)
^{-1}\left(  x\right)  \equiv u^{x_{i}}$ and%
\begin{align*}
&  u_{n+1}^{x_{1}}-u_{n+1}^{x_{2}}=u_{n}^{x_{1}}-u_{n}^{x_{2}}+A_{x_{1}}%
^{-1}\left(  E\left(  u_{n}^{x_{1}},x_{1}\right)  -x\right)  -A_{x_{2}}%
^{-1}\left(  E\left(  u_{n}^{x_{2}},x_{2}\right)  -x\right)  =\\
&  =\left\{  u_{n}^{x_{1}}-u_{n}^{x_{2}}+A_{x_{1}}^{-1}\left(  E\left(
u_{n}^{x_{1}},x_{1}\right)  -E\left(  u_{n}^{x_{2}},x_{1}\right)  \right)
\right\}  +\left\{  \left(  A_{x_{2}}^{-1}-A_{x_{1}}^{-1}\right)  \left(
x\right)  \right\}  +\\
&  +\left\{  A_{x_{1}}^{-1}\left(  E\left(  u_{n}^{x_{2}},x_{1}\right)
-E\left(  u_{n}^{x_{2}},x_{2}\right)  \right)  \right\}  +\left\{  \left(
A_{x_{1}}^{-1}-A_{x_{2}}^{-1}\right)  E\left(  u_{n}^{x_{2}},x_{2}\right)
\right\} \\
&  \equiv\left\{  A\right\}  +\left\{  B\right\}  +\left\{  C\right\}
+\left\{  D\right\}  .
\end{align*}
Now,%
\[
\left\vert \left\{  A\right\}  \right\vert =\left\vert \varphi_{x_{1}}\left(
u_{n}^{x_{1}}\right)  -\varphi_{x_{1}}\left(  u_{n}^{x_{1}}\right)
\right\vert \leq\delta\left\vert u_{n}^{x_{1}}-u_{n}^{x_{1}}\right\vert
\]
while%
\[
\left\vert \left\{  B\right\}  \right\vert +\left\vert \left\{  C\right\}
\right\vert +\left\vert \left\{  D\right\}  \right\vert \leq c\left\vert
x_{1}-x_{2}\right\vert ^{\alpha}%
\]
(where we used the fact that $E\left(  u,\overline{x}\right)  $ is $C^{\alpha
}$ in $\overline{x},$ uniformly in $u$, for small $u$). Hence%
\[
\left\vert u_{n+1}^{x_{1}}-u_{n+1}^{x_{2}}\right\vert \leq\delta\left\vert
u_{n}^{x_{1}}-u_{n}^{x_{1}}\right\vert +c\left\vert x_{1}-x_{2}\right\vert
^{\alpha}.
\]
Passing to the limit for $n\rightarrow\infty$ we get%
\[
\left\vert u^{x_{1}}-u^{x_{2}}\right\vert \leq\delta\left\vert u^{x_{1}%
}-u^{x_{2}}\right\vert +c\left\vert x_{1}-x_{2}\right\vert ^{\alpha}%
\]
Siand so%
\[
\left\vert u^{x_{1}}-u^{x_{2}}\right\vert \leq\frac{c}{1-\delta}\left\vert
x_{1}-x_{2}\right\vert ^{\alpha}%
\]
which is the desired $C^{\alpha}$ continuous dependence.

To show the $C^{\alpha}$ dependence of $\overline{x}\mapsto\frac
{\partial\left(  \Theta_{\overline{x}}\left(  x\right)  _{i}\right)
}{\partial x_{j}}$ it is enough to differentiate with respect to $x$ the
identity:
\[
x=E\left(  \Theta_{\overline{x}}\left(  x\right)  ,\overline{x}\right)
\]
finding the matrix identity%
\[
I=\frac{\partial E}{\partial u}\left(  \Theta_{\overline{x}}\left(  x\right)
,\overline{x}\right)  \frac{\partial\Theta_{\overline{x}}}{\partial x}\left(
x\right)
\]
and then%
\[
\frac{\partial\Theta_{\overline{x}}}{\partial x}\left(  x\right)  =\left(
\frac{\partial E}{\partial u}\left(  \Theta_{\overline{x}}\left(  x\right)
,\overline{x}\right)  \right)  ^{-1}.
\]
Since $\frac{\partial E}{\partial u}\left(  u,\overline{x}\right)  $ is smooth
in $x$ and $C^{\alpha}$ in $\overline{x},$ and $\Theta_{\overline{x}}\left(
x\right)  $ is $C^{\alpha}$ in $\overline{x}$, we get the desired result.
\end{proof}

We can now deduce from Proposition \ref{Prop weight X-Y} a local approximation
result, analogous to Theorem \ref{Thm Approximation local}:

\begin{theorem}
\label{Thm X=Y+R}If $X_{\left[  I\right]  }^{u}$ denotes the vector field
$X_{\left[  I\right]  }$ expressed in regularized canonical coordinates
centered at $\overline{x},$ and $Y_{\left[  I\right]  }$ are left invariant
homogeneous vector field on the group $G$, as above, then
\[
X_{\left[  I\right]  }^{u}=Y_{\left[  I\right]  }+R_{\overline{x},\left[
I\right]  },
\]
where $R_{\overline{x},\left[  I\right]  }$ is a $C^{r-\left\vert I\right\vert
,\alpha}$ vector field of weight $\geq\alpha-\left\vert I\right\vert $ near
$\overline{x},$ depending on $\overline{x}$ in a $C^{\alpha}$ continuous way.
\end{theorem}

\begin{proof}
Fix a point $\overline{x}\in\mathbb{R}^{p},$ and define the smooth
approximating vector fields $S_{i,\overline{x}}$. Here it will be more
convenient to denote by the subscript $\overline{x}$ the dependence on the
center $\overline{x}$ of the approximation$.$ We know that, expressing the
vector fields with respect to the canonical coordinates $u$ of $S_{\left[
I\right]  ,\overline{x}}$ at $\overline{x}$ (regularized canonical coordinates
of $X_{\left[  I\right]  }$ at $\overline{x}$),%
\[
X_{\left[  I\right]  }^{u}-S_{\left[  I\right]  ,\overline{x}}^{u}\text{ has
weight }\geq\alpha-\left\vert I\right\vert \text{ near }u=0,\text{ for any
}\left\vert I\right\vert \leq r\text{.}%
\]
More explicitly, this means that we can write%
\begin{equation}
X_{\left[  I\right]  }^{u}=S_{\left[  I\right]  ,\overline{x}}^{u}+O_{\left[
I\right]  ,\overline{x}}\ , \label{X=S+O}%
\end{equation}
where $O_{\left[  I\right]  ,\overline{x}}$ are $C^{r-\left\vert I\right\vert
,\alpha}$ vector fields (in the variables $u$) of weight $\geq\alpha
-\left\vert I\right\vert $ near $u=0,$ and their coefficients are $C^{\alpha}$
functions of $\overline{x}$, because the same is true for $S_{\left[
I\right]  ,\overline{x}}^{u}.$ This last assertion follows directly by the
definition of $S_{i,\overline{x}}$ and our assumptions, in view of the
following remark:
\[
\text{If }X_{i}=\sum_{j=1}^{p}b_{ij}\left(  x\right)  \partial_{x_{j}}\text{
then }S_{i,\overline{x}}=\sum_{j=1}^{p}\left(  \sum_{\left\vert \beta
\right\vert \leq r-p_{i}}\frac{\partial_{x}^{\beta}b_{ij}\left(  \overline
{x}\right)  }{\beta!}\left(  x-\overline{x}\right)  ^{\beta}\right)
\partial_{x_{j}}.
\]
Since we are assuming $b_{ij}\in C^{r-p_{i},\alpha}$, from the above formula
one reads that:

\begin{enumerate}
\item[(i)] the coefficients of $S_{i,\overline{x}}$ are $C^{\alpha}$ functions
of $\overline{x}$;

\item[(ii)] the same is true for commutators $S_{\left[  I\right]
,\overline{x}}$ for $\left\vert I\right\vert \leq r$;

\item[(iii)] the same is true if we express $S_{\left[  I\right]
,\overline{x}}$ with respect to new variables $u$ which are smooth functions
of $x$, since the Jacobian matrix of the map $\Theta_{\overline{x}}$ depends
on $\overline{x}$ in a $C^{\alpha}$ way.
\end{enumerate}

This completes the proof of $C^{\alpha}$ dependence of $O_{\left[  I\right]
,\overline{x}}$ on $\overline{x}$.

We now consider the $S_{i,\overline{x}}$'s as smooth vector fields defined in
the whole space $\mathbb{R}^{p}$ and, for any fixed $y\in\mathbb{R}^{p}$, we
apply Rothschild-Stein's local approximation theorem to the smooth
$S_{i,\overline{x}}$'s, writing%
\begin{equation}
S_{\left[  I\right]  ,\overline{x}}^{v}=Y_{\left[  I\right]  }+\widehat
{R}_{y,\left[  I\right]  }^{\overline{x}}, \label{S=Y+R}%
\end{equation}
where $Y_{\left[  I\right]  }$ are left invariant vector fields on the group,
$\widehat{R}_{y,\left[  I\right]  }^{\overline{x}}$ are smooth vector fields
of weight $\geq1-\left\vert I\right\vert ,$ smoothly depending on the point
$y,$ and the superscript $v$ in $S_{\left[  I\right]  ,\overline{x}}^{v}$
means that these vector fields are expressed with respect to the canonical
coordinates $v$ of $S_{\left[  I\right]  ,\overline{x}}$ centered at $y$. The
vector fields $\widehat{R}_{y,\left[  I\right]  }^{\overline{x}}$ also depend
on $\overline{x},$ because a different $\overline{x}$ means a different set of
vector fields $S_{\left[  I\right]  ,\overline{x}}^{v}.$ Since, by point (iii)
here above, $S_{\left[  I\right]  ,\overline{x}}^{v}$ depend on $\overline{x}$
in a $C^{\alpha}$-continuous way, the same is true for $\widehat{R}_{y,\left[
I\right]  }^{\overline{x}}.$

Next, we set $y=\overline{x}$ in (\ref{S=Y+R}); then $v=u$ (canonical
coordinates of $S_{\left[  I\right]  ,\overline{x}}$ centered at $\overline
{x}$), so we can write%
\[
S_{\left[  I\right]  ,\overline{x}}^{u}=Y_{\left[  I\right]  }+\widehat
{R}_{\overline{x},\left[  I\right]  }^{\overline{x}}%
\]
where $\widehat{R}_{\overline{x},\left[  I\right]  }^{\overline{x}}$ is a
smooth vector field of weight $\geq1-\left\vert I\right\vert $ near
$\overline{x},$ depending on $\overline{x}$ in a $C^{\alpha}$ continuous way.
This fact, together with (\ref{X=S+O}), allows us to write:%
\[
X_{\left[  I\right]  }^{u}=Y_{\left[  I\right]  }+R_{\overline{x},\left[
I\right]  }%
\]
where $R_{\overline{x},\left[  I\right]  }$ is a $C^{r-\left\vert I\right\vert
,\alpha}$ vector field of weight $\geq\alpha-\left\vert I\right\vert $ near
$\overline{x},$ depending on $\overline{x}$ in a $C^{\alpha}$ continuous way.
\end{proof}

\hfill

Finally, by Theorem \ref{Thm X=Y+R} and Proposition \ref{Prop Theta C alfa} we
immediately get:

\begin{theorem}
\label{Thm nonsmooth approx}Let $Y_{\left[  I\right]  }$ be the left invariant
homogeneous vector field on the group $G$. Then
\begin{equation}
X_{\left[  I\right]  }^{x}\left(  f\left(  \Theta_{y}\left(  x\right)
\right)  \right)  =\left(  Y_{\left[  I\right]  }f+R_{y,\left[  I\right]
}f\right)  \left(  \Theta_{y}\left(  x\right)  \right)  , \label{approx}%
\end{equation}
where $\Theta_{y}\left(  \cdot\right)  $ is a smooth diffeomorphism, depending
on $y$ in a $C^{\alpha}$ continuous way, and $R_{y,\left[  I\right]  }$ are
$C^{r-\left\vert I\right\vert ,\alpha}$ vector fields of weight $\geq
\alpha-\left\vert I\right\vert ,$ depending on $y$ in a $C^{\alpha}$
continuous way.
\end{theorem}

\subsection{Equivalent quasidistances and properties of the nonsmooth map
$\Theta$\label{section nonsmooth theta}}

Here we will prove two useful properties of the map $\Theta_{y}\left(
\cdot\right)  $.

\begin{proposition}
\label{Prop equiv dist}For every $\overline{x}\in\mathbb{R}^{p}$ there exist a
neighborhood $W$ of $\overline{x}$ and constants $C_{1},C_{2}>0$ such that for
any $x,y\in W,$ the following local equivalence holds:%
\[
C_{1}d\left(  x,y\right)  \leq\rho\left(  x,y\right)  \leq C_{2}d\left(
x,y\right)  .
\]

\end{proposition}

\begin{proof}
By the boll-box theorem for free smooth vector fields, we have%
\[
\rho\left(  x,y\right)  =\left\Vert \Theta_{y}\left(  x\right)  \right\Vert
\simeq d_{\widetilde{S}^{y}}\left(  y,x\right)  .
\]
In turn, by Proposition \ref{Proposition S_i},
\[
d_{\widetilde{S}^{y}}\left(  y,x\right)  \simeq d_{\widetilde{X}}\left(
y,x\right)
\]
and we are done.
\end{proof}

\begin{proposition}
\label{Prop nonsmooth change}The change of coordinate in $\mathbb{R}^{p}$
given by%
\[
u=\Theta_{y}\left(  x\right)
\]
has a Jacobian determinant given by%
\[
dx=c\left(  y\right)  \left(  1+O\left(  \left\Vert u\right\Vert \right)
\right)  du
\]
where $c\left(  y\right)  $ is a $C^{\alpha}$ function, bounded and bounded
away from zero. More explicitly, this means that $dx=\left[  c\left(
y\right)  +\omega\left(  y,u\right)  \right]  du$ with $\left\vert
\omega\left(  y,u\right)  \right\vert \leq c\left\Vert u\right\Vert $ and
$c\left(  y\right)  $ as above.
\end{proposition}

\begin{proof}
We will compute the Jacobian determinant of the inverse mapping $x=\Theta
_{y}^{-1}(u)$. To do this, set%
\[
X_{\left[  I\right]  }=\sum_{k=1}^{N}c_{Ik}(x)\,\frac{\partial}{\partial
x_{k}}\text{ \ \ for every }I\in B
\]
and rewrite the left hand side of (\ref{approx}) as%
\[
\sum_{k}c_{Ik}(x)\,\sum_{j}\frac{\partial f}{\partial u_{j}}\,\left(
\Theta_{y}\left(  x\right)  \right)  \text{ }\frac{\partial}{\partial y_{k}%
}\,\left[  \left(  \Theta_{y}\left(  x\right)  \right)  _{j}\right]  .
\]
Then (\ref{approx}), evaluated at $x=y$, becomes:
\[
\sum_{k}c_{Ik}(y)\,\sum_{j}\frac{\partial f}{\partial u_{j}}\left(  0\right)
\frac{\partial}{\partial y_{k}}\,\left[  \left(  \Theta_{y}\left(  x\right)
\right)  _{j}\right]  _{x=y}=\left(  Y_{\left[  I\right]  }f+R_{\left[
I\right]  }^{y}f\right)  (0).
\]
Choosing $f(u)=u_{J}$ $\left(  J\in B\right)  $,
\[
\sum_{k}c_{Ik}(y)\,\frac{\partial}{\partial y_{k}}\,\left[  \left(  \Theta
_{y}\left(  x\right)  \right)  _{J}\right]  _{y=x}=\left(  Y_{\left[
I\right]  }u_{J}+R_{\left[  I\right]  }^{y}u_{J}\right)  (0)=\delta_{IJ}%
\]
where the last equality follows recalling that
\[
Y_{\left[  I\right]  }\left[  f\right]  \left(  0\right)  =\frac{d}%
{dt}\,f\left(  \exp\,tY_{\left[  I\right]  }\right)  _{/t=0}%
\]
and that $\exp\,tY_{\left[  I\right]  }$ equals, in local coordinates,
$(0,\ldots,t,\ldots,0)$ with $t$ in the $\left[  I\right]  $-th position. As
to $R_{\left[  I\right]  }^{y}u_{J},$ by Theorem \ref{Thm nonsmooth approx} it
has weight $\geq\alpha-\left\vert I\right\vert ,$ which by definition means
that%
\[
R_{\left[  I\right]  }^{y}=\sum_{J\in B}a_{IJ}\left(  u\right)  \frac
{\partial}{\partial u_{J}}\text{ \ with }\left\vert a_{IJ}\left(  u\right)
\right\vert \leq C\left\Vert u\right\Vert ^{\alpha-\left\vert I\right\vert
+\left\vert J\right\vert }.
\]
Then%
\begin{align*}
R_{\left[  I\right]  }^{y}u_{J}  &  =\delta_{IJ}a_{IJ}\left(  u\right)  ;\\
\left\vert R_{\left[  I\right]  }^{y}u_{J}\right\vert  &  \leq C\left\Vert
u\right\Vert ^{\alpha};\\
\left(  R_{\left[  I\right]  }^{y}u_{J}\right)  \left(  0\right)   &  =0.
\end{align*}
Defining the square matrix
\[
C(y)=\left\{  c_{hk}(y)\right\}  _{hk}%
\]
and letting $J(y)$ be the Jacobian determinant of the mapping $u=(\Theta
_{y}(x))$ at $x=y$, we get
\[
\text{Det}\left[  C(y)\right]  \cdot J(y)=1.
\]
Hence the Jacobian determinant of the mapping $x=\Theta_{y}^{-1}(u)$ at $u=0$
equals Det$\left[  C(y)\right]  \equiv c(y),$ which is a $C^{\alpha}$
function, as the coefficients of the vector fields $X_{\left[  I\right]  }$
are. Moreover $c\left(  y\right)  $ is bounded away from zero since the
$X_{\left[  I\right]  }$'s are independent.

Since the determinant of $x=\Theta_{y}^{-1}(u)$ is a smooth function in $u$,
it equals%
\[
c\left(  y\right)  +\omega\left(  y,u\right)
\]
with $\left\vert \omega\left(  y,u\right)  \right\vert \leq c\left\Vert
u\right\Vert $ and we conclude%
\[
dx=c(y)\cdot\left(  1+O\left(  \left\Vert u\right\Vert \right)  \right)  du.
\]

\end{proof}

\bigskip

\bigskip

\noindent\textsc{Dipartimento di Matematica}

\noindent\textsc{Politecnico di Milano}

\noindent\textsc{Via Bonardi 9, 20133 Milano, ITALY}

\noindent\texttt{marco.bramanti@polimi.it}

\bigskip

\noindent\textsc{Dipartimento di Ingegneria dell'Informazione e Metodi
Matematici}

\noindent\textsc{Universit\`{a} di Bergamo}

\noindent\textsc{Viale Marconi 5, 24044 Dalmine BG, ITALY}

\noindent\texttt{luca.brandolini@unibg.it}

\bigskip

\noindent\textsc{Dipartimento di Ingegneria dell'Informazione e Metodi
Matematici}

\noindent\textsc{Universit\`{a} di Bergamo}

\noindent\textsc{Viale Marconi 5, 24044 Dalmine BG, ITALY}

\noindent\texttt{marco.pedroni@unibg.it}

\end{document}